\documentclass[preprint,12pt]{elsarticle}

\usepackage{amssymb}
\usepackage{amsmath}
\usepackage{amsfonts}
\usepackage{mathrsfs}
\usepackage{hyperref}
\usepackage{cases}
 \usepackage{amsthm}
\usepackage{setspace}
\begin{document}

\begin{frontmatter}



\title{Ground state solutions for the nonlinear fractional Schr\"{o}dinger-Poisson system}


\author{Kaimin Teng}

\address{Department of Mathematics, Taiyuan University of Technology,\\
Taiyuan, Shanxi 030024, P. R. China\\
E-Mail: tengkaimin2013@163.com}

\begin{abstract}
In this paper, we study the existence of ground state solutions for the nonlinear fractional Schr\"{o}dinger-Poisson system
\begin{equation*}
\left\{
  \begin{array}{ll}
    (-\Delta)^su+V(x)u+\phi u=|u|^{p-1}u, & \hbox{in $\mathbb{R}^3$,} \\
    (-\Delta)^s\phi=u^2,& \hbox{in $\mathbb{R}^3$,}
  \end{array}
\right.
\end{equation*}
where $2<p<2_s^{\ast}-1=\frac{3+2s}{3-2s}$, $s\in(\frac{3}{4},1)$. Under certain assumptions on $V$, a nontrivial ground state solution $(u,\phi)$ is established through using a monotonicity trick and global compactness Lemma. As its supplementary results, we prove some nonexistence results in the case of $1<p\leq 2$ and $p=2_s^{\ast}-1$.
\end{abstract}


\end{frontmatter}

\newtheorem{theorem}{Theorem}[section]
\newtheorem{lemma}[theorem]{Lemma}
\newtheorem{proposition}[theorem]{Proposition}
\newtheorem{definition}[theorem]{Definition}
\newtheorem{remark}[theorem]{Remark}
\newtheorem{corollary}[theorem]{Corollary}
\newtheorem{acknowledgements}[theorem]{Acknowledgements}

% \linenumbers


\section{Introduction}
In this paper we consider the following fractional Schr\"{o}dinger-Poisson system
\begin{equation}\label{main}
\left\{
  \begin{array}{ll}
    (-\Delta)^su+V(x)u+\phi u=|u|^{p-1}u, & \hbox{in $\mathbb{R}^3$,} \\
    (-\Delta)^s\phi=u^2,& \hbox{in $\mathbb{R}^3$,}
  \end{array}
\right.
\end{equation}
where $2<p<2_s^{\ast}-1=\frac{3+2s}{3-2s}$, $s\in(\frac{3}{4},1)$. We assume that the potential $V(x)$ satisfies the following conditions:\\
$(V_1)$ $V(x)\in C^1(\mathbb{R}^3,\mathbb{R})$, $(x,\nabla V(x))\in L^{\infty}(\mathbb{R}^3)\cup L^{\frac{2_s^{\ast}}{2_s^{\ast}-2}}(\mathbb{R}^3)$ and
\begin{equation*}
2sV(x)+(x,\nabla V(x))\geq0,\quad\,\,x\in\mathbb{R}^3,
\end{equation*}
where $(\cdot,\cdot)$ is the usual inner product in $\mathbb{R}^3$.\\
$(V_2)$ $V(x)\leq\liminf\limits_{|x|\rightarrow+\infty}V(x)=V_{\infty}\in\mathbb{R}^{+}$ and the inequality is strict in a subset of positive Lebesgue measure.\\
$(V_3)$ there exists a constant $\alpha_0>0$ such that
\begin{equation*}
\alpha_0=\inf_{u\in H^s(\mathbb{R}^3)\backslash\{0\}}\frac{\int_{\mathbb{R}^3}|(-\Delta)^{\frac{s}{2}}u|^2+V(x)|u|^2\,{\rm d}x}{\int_{\mathbb{R}^3}|u|^2\,{\rm d}x}>0.
\end{equation*}
Here the operator $(-\Delta)^s$ is a non-local operator, and the fractional Laplacian $(-\Delta)^s$ of order $s$ can be defined by the Fourier transform  $(-\Delta)^su=\mathcal{F}^{-1}(|\xi|^{2s}\mathcal{F}u)$, $\mathcal{F}$ being the usual Fourier transform in $\mathbb{R}^3$. In recent several years, nonlinear equations or systems involving fractional operators are receiving a great attention, because its important role is playing in the real world. For instance: Fractional Quantum Mechanics (the derivation of the Fractional Schr\"{o}dinger equation given by N. Laskin in \cite{La,La1}), pseudodifferential operators appear in many problems in Physics and Chemistry \cite{MK}, obstacle problems \cite{S}, optimization and finance \cite{CT}, etc. On the other hand, it appears in the mathematical theory itself, such as conformal geometry and minimal surfaces \cite{CM}.

When $s=1$, the system \eqref{main} reduces to the following system
\begin{equation}\label{main-1}
\left\{
  \begin{array}{ll}
    -\Delta u+V(x)u+\phi(x)u=|u|^{p-1}u, & \hbox{in $\mathbb{R}^3$,} \\
    -\Delta \phi=u^2,& \hbox{in $\mathbb{R}^3$.}
  \end{array}
\right.
\end{equation}
This is called the system of Schr\"{o}dinger-Poisson equations because it consists of a Schr\"{o}dinger equation coupled with a Poisson term. It describes
systems of identical charged particles interacting each other in the case that effects of magnetic field could be ignored and
its solution represents, in particular, a standing wave for such a system. For a deduction of this system, see e.g. \cite{BF,BF1}. The existence and multiplicity of solutions had been investigated extensively by many authors in the past several years, we refer the interested readers to see \cite{AM,AP,AR,BF,BF1,Ruiz,ZZ} and the references therein. Especially, in \cite{Ruiz,AP,ZZ}, the authors provided some new ideas to deal with variational problems with local term.

When $\phi(x)=0$, $x\in\mathbb{R}^3$, system \eqref{main} reduces to the fractional Schr\"{o}dinger equation
\begin{equation}\label{main-2}
(-\Delta)^su+V(x)u=|u|^{p-1}u, \quad x\in\mathbb{R}^3,
\end{equation}
which is of particular interest in fractional quantum mechanics in the study of particles on stochastic fields modeled by L\'{e}vy processes \cite{La,La1}. Also a detailed mathematical description of \eqref{main-2} can be found in the appendix of \cite{DPDV}. In the remarkable work of Caffarelli and Silvestre \cite{CS}, the authors express this nonlocal operator $(-\Delta)^s$ as a Dirichlet-Neumann map for a certain elliptic boundary value problem with local differential operators defined on the upper half space. After this pioneered work, the equations involving fractional operators are receiving a great attention. For the equation \eqref{main-2}, its study began from \cite{FQT} and \cite{DPV} through using variational methods. The related work can be referred to see \cite{BV,CW,DMV,DPDV,FL,Sec,SZ,Teng,TLJ} and so on.

Furthermore, if $s=1$, the equation \eqref{main-2} reduces to the classical Schr\"{o}dinger equation
\begin{equation}\label{main-3}
-\Delta u+V(x)u=|u|^{p-1}u, \quad x\in\mathbb{R}^3,
\end{equation}
which was one of the main research subjects in the last several decades. There are a large number of perfect results in the documentary, here we do not try to list all of them, we only refer to see the papers \cite{BL,Lions1,Lions2,JT} and the book \cite{Willem}.

We call the system \eqref{main} the nonlinear fractional Schr\"{o}dinger-Poisson system because it also consists of a fractional Schr\"{o}dinger equation coupled with a fractional Poisson term. From the mathematical view, the study of system \eqref{main} is very interest because it appears a nonlocal operator, namely the fractional Laplacian $(-\Delta)^s$, and this will lead to some difficulties which techniques developed for local case can not be adapted immediately \cite{Sec}, comparing to the system \eqref{main-1}. For instance, the truncation argument has to handle carefully in the present situation; The kernel of the operator $-\Delta$ is of the form $\frac{1}{|x-y|}$, but the one of fractional operator $(-\Delta)^s$ is of the form $\frac{1}{|x-y|^{3-2s}}$; Moreover, when $V(x)$ is a constant, it is well known that the ground state solution of the equation \eqref{main-3} possesses an exponential decay (see \cite{BL}), but the fractional one \eqref{main-2} is only polynomial decay which was proved in \cite{FL}.

In \cite{ZDS}, the authors studied the existence of radial solutions for system \eqref{main}. In \cite{MS}, the authors proved the semiclassical solution for system \eqref{main} and the existence of infinitely many solutions was established in \cite{Wei}. To our best of our knowledge, in the literature there are few results on the existence of ground state solutions to the problem \eqref{main}, namely couples $(u,\phi)$ which solve \eqref{main} and minimize the action functional associated to \eqref{main} among all possible solutions. The aim of our paper is devoted to studying the existence of ground state solutions for system \eqref{main}. The study of ground state solutions was started at the works of Coleman, Glazer and Martin \cite{CGM} and Berestycki and Lions \cite{BL}. After them, the existence and the profile of ground state solutions have been studied for lots of problems by many authors. Here we cannot try to review the huge bibliography.

Our main result are stated as follows.
\begin{theorem}\label{thm1-1}
Suppose that $V(x)$ satisfies $(V_1)-(V_3)$, then system \eqref{main} has a ground state solution for any $2<p<2_s^{\ast}-1$.
\end{theorem}
\begin{remark}\label{rem1-1}
We give a unified treatment for the proof of the existence of a ground state solution to problem \eqref{main} for all $p\in(2,2_s^{\ast}-1)$. Theorem \ref{thm1-1} improves Theorem 1.4 in \cite{AP} and Theorem 1.1 in \cite{ZZ} to the nonlocal case.
\end{remark}

\begin{remark}\label{rem1-2}
There are many functions verifying the assumptions $(V_1)-(V_3)$ on $V(x)$, for example, $V(x)=2-\frac{1}{1+|x|^{2s}}$.
\end{remark}

Motivated by \cite{AP,LY,ZZ}, to prove Theorem \ref{thm1-1}, first we assume that $V(x)$ is a positive constant and we look for a minimizer of the
reduced functional restricted to a suitable manifold $\mathcal{M}$ which introduced by Ruiz in \cite{Ruiz} when $s=1$. Such a manifold is consisted of the linear combination of the Pohozaev Manifold and Nehari Manifold. It has two perfect characteristics: it is a natural constraint for the reduced functional and it contains every solution of the problem (1).
As we know, when $s=1$ and for any $3<p<5$, it can be obtained the boundedness of minimizing sequence on the Nehari manifold. It is difficult to get the boundedness of any $(PS)$ sequence when $2<p\leq3$. But on the manifold $\mathcal{M}$, we can show that any minimizing sequence is bounded for any $2<p<2_s^{\ast}$.

When $V(x)$ is not a constant, it is more difficult to the boundedness of any $(PS)$ sequence. To overcome this difficulty, we use a subtle approach developed by Jeanjean \cite{Jean}.
\begin{theorem}\label{thm1-2}
Let $X$ be a Banach space and $\Lambda\subset\mathbb{R}^{+}$ an interval. Consider a family $\varphi_{\lambda}$ of $C^1$ functionals on $X$ with the form
\begin{equation*}
\varphi_{\lambda}(u)=A(u)-\lambda B(u),\quad \forall\lambda\in\Lambda,
\end{equation*}
where $B(u)\geq0$, $\forall u\in X$, and such that either $A(u)\rightarrow+\infty$ or $B(u)\rightarrow+\infty$ as $\|u\|\rightarrow\infty$. If there exists $v_1,v_2\in X$ such that
\begin{equation*}
c_{\lambda}=\inf_{\gamma\in\Gamma}\max_{t\in[0,1]}\varphi_{\lambda}(\gamma(t))>\max\{\varphi_{\lambda}(v_1),\varphi_{\lambda}(v_2)\},\quad \forall\lambda\in\Lambda
\end{equation*}
where $\Gamma=\{\gamma\in C([0,1],X): \gamma(0)=v_1,\gamma(1)=v_2\}$.

Then for almost every $\lambda\in\Lambda$, there exists a sequence $\{v_n\}\subset X$ such that\\
$(i)$ $\{v_n\}$ is bounded;\\
$(ii)$ $\varphi_{\lambda}(v_n)\rightarrow c_{\lambda}$;\\
$(iii)$ $\varphi_{\lambda}'(v_n)\rightarrow0$ in the dual $X'$ of $X$.
\end{theorem}
For applying this theorem to system \eqref{main}, the main idea is to introduce a family of functionals $\varphi_{\lambda}$, if it satisfies all the conditions of Theorem \ref{thm1-2}, directly, we can obtain a bounded $(PS)_{c_{\lambda}}$ sequence. Through using a global compactness Lemma which is related to the functionals $\varphi_{\lambda}$ and its limit functional $\varphi_{\infty}$, we can deduce that $(PS)_{c_{\lambda}}$ condition holds, of course, before applying the global compactness Lemma, we have to consider the ground state solution for limit problem
\begin{equation}\label{main-4}
\left\{
  \begin{array}{ll}
    ги-\Delta u)^s+V_{\infty}u+\phi u=|u|^{p-1}u, & \hbox{in $\mathbb{R}^3$,} \\
    (-\Delta)^s \phi=u^2,& \hbox{in $\mathbb{R}^3$.}
  \end{array}
\right.
\end{equation}
We will establish the following result.
\begin{theorem}\label{thm1-3}
For $2<p<2_s^{\ast}-1$, problem \eqref{main-4} has a ground state solution.
\end{theorem}
Finally, choosing a sequence $\lambda_n\rightarrow1$, we can prove that $\{u_{\lambda_n}\}$ is a bounded $(PS)_{c_1}$ sequence for $\varphi_1$. Applying the global compactness Lemma again, it can yield Theorem \ref{thm1-1}.

We also obtain some nonexistence results.
\begin{theorem}\label{thm1-4}
If $1<p\leq2$, then for any $\lambda\geq\frac{1}{4}$, system
\begin{equation}\label{main-5}
\left\{
  \begin{array}{ll}
    (-\Delta)^s u+u+\lambda\phi(x)u=|u|^{p-1}u, & \hbox{in $\mathbb{R}^3$,} \\
    (-\Delta)^s \phi=u^2,& \hbox{in $\mathbb{R}^3$}
  \end{array}
\right.
\end{equation}
has no solution.
\end{theorem}
\begin{theorem}\label{thm1-5}
If $p=2_s^{\ast}$, and $V(x)$ satisfies $(V_1)-(V_2)$ or be a positive constant, then problem \eqref{main} has no solution.
\end{theorem}

The paper is organized as follows. In section 2, we present some preliminaries results, such as: we will study the properties of the local term $\phi u$ in system \eqref{main}, regularity for system \eqref{main} and Pohozaev identity will be proved. In section 3, we will prove Theorem \ref{thm1-3}. In section 4, we will establish a global compactness lemma, and our main result--Theorem \ref{thm1-1} will be proved. In section 5, we will prove Theorem \ref{thm1-4} and \ref{thm1-5}.
\section{Preliminaries}

In this section, we outline the variational framework for studying problem \eqref{main}, investigate some properties of the local term $\phi u$ appearing in \eqref{main}, and establish the Pohozaev identity of system \eqref{main}. In the sequel, $\alpha$ will denote a fixed number, $\alpha\in(0,1)$, we denote by $\|\cdot\|_{p}$ the usual norm of the space $L^p(\mathbb{R}^N)$, the letter $c_i$ ($i=1,2,\ldots$) or $C$ denote by some positive constants. For simplicity, We denote $\widehat{u}$ the Fourier transform of $u$.

\subsection{Work space}
We define the homogeneous fractional Sobolev space $\mathcal{D}^{\alpha,2}(\mathbb{R}^3)$ as follows
\begin{equation*}
\mathcal{D}^{\alpha,2}(\mathbb{R}^3)=\Big\{u\in L^{2_{\alpha}^{\ast}}(\mathbb{R}^3)\,\,\Big|\,\,|\xi|^{\alpha}\widehat{u}(\xi)\in L^2(\mathbb{R}^3)\Big\}
\end{equation*}
which is the completion of $C_0^{\infty}(\mathbb{R}^3)$ under the norm
\begin{equation*}
\|u\|_{\mathcal{D}^{\alpha,2}}=\int_{\mathbb{R}^3}|(-\Delta)^{\frac{\alpha}{2}}u|^2\,{\rm d}x=\int_{\mathbb{R}^3}|\xi|^{2\alpha}|\widehat{u}(\xi)|^2\,{\rm d}\xi
\end{equation*}
The fractional Sobolev space $H^{\alpha}(\mathbb{R}^3)$ can be described by means of the Fourier transform, i.e.
\begin{equation*}
H^{\alpha}(\mathbb{R}^3)=\Big\{u\in L^2(\mathbb{R}^3)\,\,\Big|\,\,\int_{\mathbb{R}^3}(|\xi|^{2\alpha}|\widehat{u}(\xi)|^2+|\widehat{u}(\xi)|^2)\,{\rm d}\xi<+\infty\Big\}.
\end{equation*}
In this case, the inner product and the norm are defined as
\begin{equation*}
(u,v)=\int_{\mathbb{R}^3}(|\xi|^{2\alpha}\widehat{u}(\xi)\overline{\widehat{v}(\xi)}+\widehat{u}(\xi)\overline{\widehat{v}(\xi)})\,{\rm d}\xi
\end{equation*}
\begin{equation*}
\|u\|_{H^{\alpha}}=\bigg(\int_{\mathbb{R}^3}(|\xi|^{2\alpha}|\widehat{u}(\xi)|^2+|\widehat{u}(\xi)|^2)\,{\rm d}\xi\bigg)^{\frac{1}{2}},
\end{equation*}
From Plancherel's theorem we have $\|u\|_2=\|\widehat{u}\|_2$ and $\||\xi|^{\alpha}\widehat{u}\|_2=\|(-\Delta)^{\frac{\alpha}{2}}u\|_2$. Hence
\begin{equation*}
\|u\|_{H^{\alpha}}=\bigg(\int_{\mathbb{R}^3}(|(-\Delta)^{\frac{\alpha}{2}}u(x)|^2+|u(x)|^2)\,{\rm d}x\bigg)^{\frac{1}{2}},\quad \forall u\in H^{\alpha}(\mathbb{R}^3).
\end{equation*}

In terms of finite differences, the fractional Sobolev space $H^{\alpha}(\mathbb{R}^3)$ also can be defined as follows
\begin{equation*}
H^{\alpha}(\mathbb{R}^3)=\Big\{u\in L^2(\mathbb{R}^3)\,\,\Big|\,\,\frac{|u(x)-u(y)|}{|x-y|^{\frac{3}{2}+\alpha}}\in L^2(\mathbb{R}^3\times\mathbb{R}^3)\Big\}
\end{equation*}
endowed with the natural norm
\begin{equation*}
\|u\|_{H^{\alpha}}=\bigg(\int_{\mathbb{R}^3}|u|^2\,{\rm d}x+\int_{\mathbb{R}^3}\int_{\mathbb{R}^3}\frac{|u(x)-u(y)|^2}{|x-y|^{3+2\alpha}}\,{\rm d}x\,{\rm d}y\bigg)^{\frac{1}{2}}.
\end{equation*}

Also, in light of Proposition 3.4 and Proposition 3.6 in \cite{NPV}, we have
\begin{equation*}
\|(-\Delta)^{\frac{\alpha}{2}}u\|_2^2=\int_{\mathbb{R}^3}|\xi|^{2\alpha}|\widehat{u}(\xi)|^2\,{\rm d}\xi=\frac{1}{C(\alpha)}\int_{\mathbb{R}^3}\int_{\mathbb{R}^3}\frac{|u(x)-u(y)|^2}{|x-y|^{3+2\alpha}}\,{\rm d}x\,{\rm d}y.
\end{equation*}

It is well known that $H^{\alpha}(\mathbb{R}^3)$ is continuously embedded into $L^p(\mathbb{R}^3)$ for $2\leq p\leq 2_{\alpha}^{\ast}$ ($2_{\alpha}^{\ast}=\frac{6}{3-2\alpha}$), and for any $\alpha\in(0,1)$, there exists a best constant $S_{\alpha}>0$ such that
\begin{equation}\label{equ2-1}
S_{\alpha}=\inf_{u\in \mathcal{D}^{\alpha,2}}\frac{\int_{\mathbb{R}^3}|(-\Delta)^{\frac{\alpha}{2}}u|^2\,{\rm d}x}{\Big(\int_{\mathbb{R}^3}|u(x)|^{2_{\alpha}^{\ast}}\,{\rm d}x\Big)^{\frac{2}{2_{\alpha}^{\ast}}}}.
\end{equation}
We recall some fundamental facts in the Sobolev space in $H^s(\mathbb{R}^3)$.
\begin{lemma}\label{lemp-2-1}(Cut-off estimate, \cite{AS}).
Let $u\in H^{\alpha}(\mathbb{R}^3)$ and $\varphi\in C_0^{\infty}(\mathbb{R}^3)$ with $0\leq \varphi\leq1$ and $|\varphi'|\leq L$. Then, for every pair of measurable sets $\Omega_1,\Omega_2\subset\mathbb{R}^3$, we have
\begin{align*}
\int_{\Omega_1\times\Omega_2}\frac{|\varphi(x)u(x)-\varphi(y)u(y)|^2}{|x-y|^{3+2\alpha}}\,{\rm d}x\,{\rm d}y\leq& C\min\Big\{\int_{\Omega_1}|u(x)|^2\,{\rm d}x,\,\,\int_{\Omega_2}|u(x)|^2\,{\rm d}x\Big\}\\
&+C\int_{\Omega_1\times\Omega_2}\frac{|u(x)-u(y)|^2}{|x-y|^{3+2\alpha}}\,{\rm d}x\,{\rm d}y
\end{align*}
where $C$ depends on $s$ and the constant $L$.
\end{lemma}
For any measurable function $u$ consider the corresponding symmetric radial decreasing rearrangement $u^{\ast}$, whose classical definition and basic properties
can be found, for instance, in \cite{BF}. We recall the fractional the Polya-Szeg\"{o} inequality in \cite{DPV}.
\begin{lemma}\label{lemp-2-2}
For any $u\in H^{\alpha}(\mathbb{R}^3)$, the following inequality holds
\begin{equation*}
\int_{\mathbb{R}^3}|(-\Delta )^{\frac{\alpha}{2}}u^{\ast}|^2\,{\rm d}x\leq \int_{\mathbb{R}^3}|(-\Delta )^{\frac{\alpha}{2}}u|^2\,{\rm d}x
\end{equation*}
or equivalently
\begin{equation*}
\int_{\mathbb{R}^3}\int_{\mathbb{R}^3}\frac{|u^{\ast}(x)-u^{\ast}(y)|^2}{|x-y|^{3+2\alpha}}\,{\rm d}x\,{\rm d}y\leq\int_{\mathbb{R}^3}\int_{\mathbb{R}^3}\frac{|u(x)-u(y)|^2}{|x-y|^{3+2\alpha}}\,{\rm d}x\,{\rm d}y.
\end{equation*}
\end{lemma}

\begin{lemma}\label{lemp-2-3}(Riesz's inequality)
Let $H$ be a nonincreasing function on the positive real line with $H(t)\rightarrow0$ as $t\rightarrow+\infty$. Then, for any pair $f$, $g$ of nonnegative measurable
functions on $\mathbb{R}^N$ that vanish at infinity,
\begin{equation*}
\int_{\mathbb{R}^N}\int_{\mathbb{R}^N}f(x)g(y)H(|x-y|)\,{\rm d}x\,{\rm d}y\leq\int_{\mathbb{R}^N}\int_{\mathbb{R}^N}f^{\ast}(x)g^{\ast}(y)H(|x-y|)\,{\rm d}x\,{\rm d}y.
\end{equation*}
If $H$ is strictly decreasing, then equality (with a finite and nonzero value of the integral) occurs only if there exists a translation $\sigma$ such that $f=f^{\ast}\circ\sigma$ and $g=g^{\ast}\circ\sigma$ almost everywhere.
\end{lemma}

\subsection{Formulation of Problem \eqref{main}}

By hypotheses $(V_2)$ and $(V_3)$, fractional Sobolev space $H^s(\mathbb{R}^3)$ can be equipped with the inner product
\begin{equation*}
\langle u,v\rangle=\int_{\mathbb{R}^3}\int_{\mathbb{R}^3}\frac{(u(x)-u(y))(v(x)-v(y))}{|x-y|^{3+2s}}\,{\rm d}x\,{\rm d}y+\int_{\mathbb{R}^3}V(x)uv\,{\rm d}x
\end{equation*}
and the corresponding norm
\begin{equation*}
\|u\|=\Big(\int_{\mathbb{R}^3}|(-\Delta )^{\frac{s}{2}}u|^2+V(x)u^2\,{\rm d}x\Big)^{\frac{1}{2}}.
\end{equation*}
Indeed, from the hypotheses $(V_2)$ and $(V_3)$, the above norm is equivalent to the usual norm $\|\cdot\|_{H^s}$. In fact, from $(V_3)$, similar to the proof of Lemma 3.4 in \cite{Jean}, there exists a constant $C>0$ such that
\begin{equation*}
\int_{\mathbb{R}^3}|(-\Delta )^{\frac{s}{2}}u|^2+V(x)u^2\,{\rm d}x\geq\frac{\alpha_0}{2}\int_{\mathbb{R}^3}|u|^2\,{\rm d}x+C\int_{\mathbb{R}^3}|(-\Delta )^{\frac{s}{2}}u|^2\,{\rm d}x
\end{equation*}
By $(V_2)$, we have that
\begin{equation*}
\int_{\mathbb{R}^3}|(-\Delta )^{\frac{s}{2}}u|^2+V(x)u^2\,{\rm d}x\leq\int_{\mathbb{R}^3}|(-\Delta )^{\frac{s}{2}}u|^2\,{\rm d}x+\int_{\mathbb{R}^3}V_{\infty}u^2\,{\rm d}x
\end{equation*}
The above two estimates imply that $\|\cdot\|$ is an equivalent norm on $H^s(\mathbb{R}^3)$.

It is easy to show that problem \eqref{main} can be reduced to a single fractional Schr\"{o}dinger equation with a nonlocal term. Actually, consider $u\in H^{s}(\mathbb{R}^3)$, the linear functional $\mathcal{L}_u$ defined in $\mathcal{D}^{s,2}(\mathbb{R}^3)$ by
\begin{equation*}
\mathcal{L}_u(v)=\int_{\mathbb{R}^3}u^2v\,{\rm d}x,
\end{equation*}
the H\"{o}lder's inequality and \eqref{equ2-1} implies that
\begin{align}\label{equ2-2}
|\mathcal{L}_u(v)|&\leq\Big(\int_{\mathbb{R}^3}|u(x)|^{\frac{12}{3+2s}}\,{\rm d}x\Big)^{\frac{3+2s}{6}}\Big(\int_{\mathbb{R}^3}|v(x)|^{2_s^{\ast}}\,{\rm d}x\Big)^{\frac{1}{2_s^{\ast}}}\nonumber\\
&\leq S_s^{\frac{1}{2}}\Big(\int_{\mathbb{R}^3}|u(x)|^{\frac{12}{3+2s}}\,{\rm d}x\Big)^{\frac{3+2s}{6}}\|v\|_{\mathcal{D}^{s,2}}\leq S_s^{\frac{1}{2}} C\|u\|_{H^s}^2\|v\|_{\mathcal{D}^{s,2}},
\end{align}
where using the following fact that $H^s(\mathbb{R}^3)\hookrightarrow L^{\frac{12}{3+2s}}(\mathbb{R}^3)$ if $s>\frac{1}{2}$. Hence, by the Lax-Milgram theorem, there exists a unique $\phi_u^s\in\mathcal{D}^{s,2}(\mathbb{R}^3)$ such that
\begin{equation}\label{equ2-3}
\int_{\mathbb{R}^3}(-\Delta)^{\frac{s}{2}}\phi_u^s(-\Delta)^{\frac{s}{2}}v\,{\rm d}x=\int_{\mathbb{R}^3}u^2v\,{\rm d}x,\quad \forall v\in\mathcal{D}^{s,2}(\mathbb{R}^3),
\end{equation}
that is $\phi_u^s$ is a weak solution of
\begin{equation}\label{equ2-4}
(-\Delta)^s\phi_u^s=u^2,\quad x\in\mathbb{R}^3
\end{equation}
and the representation formula holds
\begin{equation}\label{equ2-5}
\phi_u^s(x)=c_s\int_{\mathbb{R}^3}\frac{u^2(y)}{|x-y|^{3-2s}}\,{\rm d}y,\quad x\in\mathbb{R}^3,
\end{equation}
which is called $t$-Riesz potential, where
\begin{equation*}
c_s=\pi^{-\frac{3}{2}}2^{-2s}\frac{\Gamma(3-2s)}{\Gamma(s)}.
\end{equation*}
It follows from \eqref{equ2-5} that $\phi_u^s\geq0$ for all $x\in\mathbb{R}^3$. Combining \eqref{equ2-2} and \eqref{equ2-3}, we have
\begin{align*}
\|\phi_u^s\|_{\mathcal{D}^{s,2}}^2&=\int_{\mathbb{R}^3}\phi_u^su^2\,{\rm d}x\leq\Big(\int_{\mathbb{R}^3}|u(x)|^{\frac{12}{3+2s}}\,{\rm d}x\Big)^{\frac{3+2s}{6}}\Big(\int_{\mathbb{R}^3}|\phi_u^s(x)|^{2_s^{\ast}}\,{\rm d}x\Big)^{\frac{1}{2_s^{\ast}}}\\
&\leq S_s^{\frac{1}{2}}\Big(\int_{\mathbb{R}^3}|u(x)|^{\frac{12}{3+2s}}\,{\rm d}x\Big)^{\frac{3+2s}{6}}\|\phi_u^s\|_{\mathcal{D}^{s,2}}\leq S_s^{\frac{1}{2}} C\|u\|_{H^s}^2\|\phi_u^s\|_{\mathcal{D}^{s,2}},
\end{align*}
that is
\begin{equation}\label{equ2-6}
\|\phi_u^s\|_{\mathcal{D}^{s,2}}\leq S_s^{\frac{1}{2}}\Big(\int_{\mathbb{R}^3}|u(x)|^{\frac{12}{3+2s}}\,{\rm d}x\Big)^{\frac{3+2s}{6}}\leq C\|u\|_{H^s}^2,\quad \text{if}\,\, s>\frac{1}{2}.
\end{equation}
Hence, by H\"{o}lder's inequality and \eqref{equ2-6}, we get
\begin{align*}
\int_{\mathbb{R}^3}\phi_u^su^2\,{\rm d}x&\leq S_s^{\frac{1}{2}}\Big(\int_{\mathbb{R}^3}|u(x)|^{\frac{12}{3+2s}}\,{\rm d}x\Big)^{\frac{3+2s}{6}}\|\phi_u^s\|_{\mathcal{D}^{s,2}}\leq C\|u\|_{H^s}^2\|\phi_u^s\|_{\mathcal{D}^{s,2}}\leq C\|u\|_{H^s}^4
\end{align*}
that is
\begin{equation}\label{equ2-7}
\int_{\mathbb{R}^3}\phi_u^su^2\,{\rm d}x\leq C\|u\|_{H^s}^4,\quad\text{if}\,\, s>\frac{1}{2}.
\end{equation}
Substituting $\phi_u^s$ in \eqref{main}, we get the following fractional Schr\"{o}dinger equation
\begin{equation}\label{equ2-8}
(-\Delta)^su+V(x)u+\phi_u^su=|u|^{p-1}u,\quad x\in\mathbb{R}^3,
\end{equation}
whose solutions can be obtained by looking for critical points of the functional $I: H^s(\mathbb{R}^3)\rightarrow\mathbb{R}$ defined by
\begin{equation*}
I(u)=\frac{1}{2}\int_{\mathbb{R}^3}(|(-\Delta)^{\frac{s}{2}}u|^2+V(x)u^2)\,{\rm d}x++\frac{1}{4}\int_{\mathbb{R}^3}\phi_u^su^2\,{\rm d}x-\frac{1}{p+1}\int_{\mathbb{R}^3}|u(x)|^{p+1}\,{\rm d}x.
\end{equation*}
Obviously, $I$ is well defined in $H^s(\mathbb{R}^3)$ and $I\in C^1(H^s(\mathbb{R}^3),\mathbb{R})$, and that
\begin{equation*}
\langle I'(u),v\rangle=\int_{\mathbb{R}^3}\Big((-\Delta)^{\frac{s}{2}}u(-\Delta)^{\frac{s}{2}}v+V(x)uv+\phi_u^suv-|u|^{p-1}uv\Big)\,{\rm d}x.
\end{equation*}
Obviously, the critical points of $I$ are the weak solutions of problem \eqref{equ2-8}.
\begin{definition}\label{def2-1}
(1) We call $(u,\phi)\in H^s(\mathbb{R}^3)\times \mathcal{D}^{s,2}(\mathbb{R}^3)$ is a weak solution of problem \eqref{main} if $u$ is a weak solution of problem \eqref{equ2-8}.\\
(2) We call $u\in H^s(\mathbb{R}^3)$ is a weak solution of \eqref{equ2-8} if
\begin{equation*}
\int_{\mathbb{R}^3}\Big((-\Delta)^{\frac{s}{2}}u(-\Delta)^{\frac{s}{2}}v+V(x)uv+\phi_u^suv-|u|^{p-1}uv\Big)\,{\rm d}x=0,\quad \forall v\in H^s(\mathbb{R}^3).
\end{equation*}
\end{definition}
Let us now define the operator $\Phi:H^s(\mathbb{R}^3)\rightarrow \mathcal{D}^{s,2}(\mathbb{R}^3)$ as
\begin{equation*}
\Phi(u)=\phi_u^s.
\end{equation*}
In the following lemma we summarize some properties of $\Phi$ which can be easily to check, useful to study our problem.

\begin{lemma}\label{lem2-1}
If $s\in(\frac{1}{2},1)$, then for any $u\in H^s(\mathbb{R}^3)$, we have\\
(1) $\Phi$ is continuous;\\
(2) $\Phi$ maps bounded sets into bounded sets;\\
(3) $\Phi(\tau u)=\tau^2\Phi(u)$ for all $\tau\in\mathbb{R}$, $\Phi(u(\cdot+y))=(\Phi(u))(x+y)$;\\
(4) $\Phi(u_{\theta})=\theta^{2s}(\Phi(u))_{\theta}$ for any $\theta>0$, where $u_{\theta}=u(\cdot/\theta)$;\\
(5) If $u_n\rightharpoonup u$ in $H^s(\mathbb{R}^3)$ then $\Phi(u_n)\rightharpoonup\Phi(u)$ in $\mathcal{D}^{s,2}(\mathbb{R}^3)$;\\
(6) If $u_n\rightarrow u$ in $H^s(\mathbb{R}^3)$, then $\Phi(u_n)\rightarrow\Phi(u)$ in $\mathcal{D}^{s,2}(\mathbb{R}^3)$ and $\int_{\mathbb{R}^3}\phi_{u_n}^su_n^2\,{\rm d}x\rightarrow \int_{\mathbb{R}^3}\phi_u^su^2\,{\rm d}x$.
\end{lemma}

Let us define a function $\Psi:H^s(\mathbb{R}^3)\rightarrow\mathbb{R}$ by
\begin{equation*}
\Psi(u)=\int_{\mathbb{R}^3}\phi_u^s(x)u^2(x)\,{\rm d}x.
\end{equation*}
It is clearly that $\Psi(u(\cdot+y))=\Psi(u)$, for any $y\in\mathbb{R}^3$, $u\in H^s(\mathbb{R}^3)$ and $\Psi$ is weakly lower semi-continuous in $H^s(\mathbb{R}^3)$.

The next lemma shows that the functional $\Psi$ and its derivative $\Psi'$ have the B-L splitting property, which is similar to the
well-known Brezis-Lieb Lemma.

\begin{lemma}\label{lem2-2}
If $u_n\rightharpoonup u$ in $H^s(\mathbb{R}^3)$ with $s\in(\frac{3}{4},1)$ and $u_n\rightarrow u$ a.e. in $\mathbb{R}^3$, then\\
$(i)$ $\Psi(u_n-u)=\Psi(u_n)-\Psi(u)+o(1)$;\\
$(ii)$ $\Psi'(u_n-u)=\Psi'(u_n)-\Psi'(u)+o(1)$ in $(H^s(\mathbb{R}^3))'$.
\end{lemma}
\begin{proof}
$(i)$ Set
\begin{equation*}
A=\int_{\mathbb{R}^3}\int_{\mathbb{R}^3}\frac{u^2(y)u^2(x)}{|x-y|^{3-2s}}\,{\rm d}x\,{\rm d}y
\end{equation*}
and
\begin{equation*}
A_n^{(1)}=\int_{\mathbb{R}^3}\int_{\mathbb{R}^3}\frac{u_n^2(y)u^2(x)}{|x-y|^{3-2s}}\,{\rm d}x\,{\rm d}y,\quad A_n^{(2)}=\int_{\mathbb{R}^3}\int_{\mathbb{R}^3}\frac{u_n(y)u(y)u_n(x)u(x)}{|x-y|^{3-2s}}\,{\rm d}x\,{\rm d}y,
\end{equation*}
\begin{equation*}
A_n^{(3)}=\int_{\mathbb{R}^3}\int_{\mathbb{R}^3}\frac{u_n^2(y)u_n(x)u(x)}{|x-y|^{3-2s}}\,{\rm d}x\,{\rm d}y,\quad A_n^{(4)}=\int_{\mathbb{R}^3}\int_{\mathbb{R}^3}\frac{u_n(y)u(y)u^2(x)}{|x-y|^{3-2s}}\,{\rm d}x\,{\rm d}y.
\end{equation*}
It is easy to check that
\begin{equation*}
\Psi(u_n-u)-(\Psi(u_n)-\Psi(u))=2A_n^{(1)}+4A_n^{(2)}-4A_n^{(3)}-4A_n^{(4)}+2A
\end{equation*}
and thus it is suffice to show that
\begin{equation}\label{equ2-10}
\lim_{n\rightarrow\infty}A_n^{(i)}=A,\quad i=1,2,3,4.
\end{equation}
Set
\begin{equation}\label{equ2-11}
w_n(x)=\int_{\mathbb{R}^3}\frac{u_n^2(y)}{|x-y|^{3-2s}}\,{\rm d}y,\quad w(x)=\int_{\mathbb{R}^3}\frac{u^2(y)}{|x-y|^{3-2s}}\,{\rm d}y.
\end{equation}
Since $u^2\in L^{\frac{6}{3+2s}}(\mathbb{R}^3)=L^{(2_s^{\ast})'}(\mathbb{R}^3)$ and by (5) of Lemma \ref{lem2-1} it holds $w_n\rightharpoonup w$ in $L^{2_s^{\ast}}(\mathbb{R}^3)$, we conclude that
\begin{equation*}
A_n^{(1)}=\int_{\mathbb{R}^3}w_n(x)u^2(x)\,{\rm d}x\rightarrow\int_{\mathbb{R}^3}w(x)u^2(x)\,{\rm d}x=A.
\end{equation*}

For $i=2$, set
\begin{equation*}
\sigma_n(x)=\int_{\mathbb{R}^3}\frac{u_n(y)u(y)}{|x-y|^{3-2s}}\,{\rm d}y.
\end{equation*}
First we show that $\sigma_n(x)\rightarrow w(x)$ a.e. $x\in\mathbb{R}^3$. Choose $p<\frac{3}{3-2s}$ and $q>\frac{3}{3-2s}$, owing to $s>\frac{3}{4}$, it is easy to check that $2p'\in(2,2_s^{\ast})$ and $2q'\in(2,2_s^{\ast})$, by H\"{o}lder's inequality, we deduce that
\begin{align}\label{equ2-12}
|\sigma_n(x)-w(x)|&\leq\int_{\mathbb{R}^3}\frac{|u_n(y)u(y)-u^2(y)|}{|x-y|^{3-2s}}\,{\rm d}y\nonumber\\
&\leq\|u_n-u\|_{L^{2p'}(B_R(x))}\|u\|_{L^{2p'}(B_R(x))}\Big(\int_{|y-x|<R}\frac{1}{|y-x|^{p(3-2s)}}\,{\rm d}y\Big)^{\frac{1}{p}}\nonumber\\
&+\|u_n-u\|_{L^{2q'}(B_R^c(x))}\|u\|_{L^{2q'}(B_R^c(x))}\Big(\int_{|y-x|\geq R}\frac{1}{|y-x|^{q(3-2s)}}\,{\rm d}y\Big)^{\frac{1}{q}}\nonumber\\
&\leq {\rm o}(1)+C\varepsilon,
\end{align}
which implies that the pointwise convergence. Moreover by the Sobolev embedding, we have
\begin{equation*}
\|\sigma_nu_n\|_{L^2}^2\leq\|\sigma_n\|_{L_{2_s^{\ast}}}^2\|u_n\|_{L^{\frac{3}{s}}}^2\leq C\|u_n\|^4\|u\|^2 \leq C
\end{equation*}
and thus, up to a subsequence, $\sigma_nu_n\rightharpoonup wu$ in $L^2(\mathbb{R}^3)$. Since $u\in L^2(\mathbb{R}^3)$, we get
\begin{equation*}
A_n^{(2)}=\int_{\mathbb{R}^3}\sigma_n(x)u_n(x)u(x)\,{\rm d}x\rightarrow\int_{\mathbb{R}^3}w(x)u^2(x)\,{\rm d}x=A.
\end{equation*}
In a similar way, we can verify \eqref{equ2-10} with $i=3,4$. Thus conclusion $(i)$ is proved.

$(ii)$ For any $\varphi\in H^s(\mathbb{R}^3)$ with $\|\varphi\|\leq1$, we need to show
\begin{equation*}
\langle\Psi'(u_n-u)-\Psi'(u_n)+\Psi'(u),\varphi\rangle\rightarrow 0\quad \text{uniformly with respect to}\,\,\varphi.
\end{equation*}
Indeed, let
\begin{equation*}
B_n^{(1)}=\int_{\mathbb{R}^3}\int_{\mathbb{R}^3}\frac{u_n^2(y)u(x)\varphi(x)}{|y-x|^{3-2s}}\,{\rm d}y\,{\rm d}x,\quad B_n^{(2)}=\int_{\mathbb{R}^3}\int_{\mathbb{R}^3}\frac{u_n(y)u(y)u_n(x)\varphi(x)}{|y-x|^{3-2s}}\,{\rm d}y\,{\rm d}x
\end{equation*}
and
\begin{equation*}
B_n^{(3)}=\int_{\mathbb{R}^3}\int_{\mathbb{R}^3}\frac{u_n(y)u(y)u(x)\varphi(x)}{|y-x|^{3-2s}}\,{\rm d}y\,{\rm d}x,\quad B_n^{(4)}=\int_{\mathbb{R}^3}\int_{\mathbb{R}^3}\frac{u^2(y)u_n(x)\varphi(x)}{|y-x|^{3-2s}}\,{\rm d}y\,{\rm d}x.
\end{equation*}
Then, by computation, we have
\begin{equation*}
\langle\Psi'(u_n-u)-\Psi'(u_n)+\Psi'(u),\varphi\rangle=-B_n^{(1)}-2B_n^{(2)}+2B_n^{(3)}+B_n^{(4)}.
\end{equation*}
We claim
\begin{equation}\label{equ2-13}
\lim_{n\rightarrow\infty}B_n^{(i)}=\int_{\mathbb{R}^3}\int_{\mathbb{R}^3}\frac{u^2(y)u(x)\varphi(x)}{|y-x|^{3-2s}}\,{\rm d}y\,{\rm d}x,\quad i=1,2,3,4,
\end{equation}
uniformly with respect to $\varphi$. We only need to verify $i=1,2$, and $i=3,4$ can be done in a similar way.

First,
\begin{align*}
\Big|B_n^{(1)}-\int_{\mathbb{R}^3}\int_{\mathbb{R}^3}\frac{u^2(y)u(x)\varphi(x)}{|y-x|^{3-2s}}\,{\rm d}y\,{\rm d}x\Big|&\leq\int_{\mathbb{R}^3}|(w_n(x)-w(x))u(x)\varphi(x)|\,{\rm d}x\\
&\leq\|\varphi\|_{L^2}\Big(\int_{\mathbb{R}^3}|w_n(x)-w(x)|^2u^2(x)\,{\rm d}x\Big)^{\frac{1}{2}}
\end{align*}
where $w_n$ and $w$ are defined by \eqref{equ2-11}. From (5) of Lemma \ref{lem2-1} and Sobolev embedding inequality, it follows that
\begin{equation*}
\int_{\mathbb{R}^3}|w_n(x)-w(x)|^{2_s^{\ast}}\,{\rm d}x\leq C\|\Phi(u_n)-\Phi(u)\|_{\mathcal{D}^{s,2}(\mathbb{R}^3)}\leq C.
\end{equation*}
For $R>0$ large enough and $n$ large enough, similar to \eqref{equ2-12}, we have
\begin{align*}
|w_n(x)-w(x)|&\leq\int_{\mathbb{R}^3}\frac{|u_n^2(y)-u^2(y)|}{|x-y|^{3-2s}}\,{\rm d}y\\
&\leq\|u_n-u\|_{L^{2p'}(B_R(x))}\|u_n+u\|_{L^{2p'}(B_R(x))}\Big(\int_{|y-x|<R}\frac{1}{|y-x|^{p(3-2s)}}\,{\rm d}y\Big)^{\frac{1}{p}}\\
&+\|u_n-u\|_{L^{2q'}(B_R^c(x))}\|u_n+u\|_{L^{2q'}(B_R^c(x))}\Big(\int_{|y-x|\geq R}\frac{1}{|y-x|^{q(3-2s)}}\,{\rm d}y\Big)^{\frac{1}{q}}\\
&\leq {\rm o}(1)+C\varepsilon,
\end{align*}
which implies that $w_n(x)\rightarrow w(x)$ a.e. $x\in\mathbb{R}^3$. Therefore, up to a subsequence, $|w_n-w|^2\rightharpoonup0$ in $L^{\frac{3}{3-2s}}(\mathbb{R}^3)$. Since $u^2\in L^{\frac{3}{2s}}(\mathbb{R}^3)$, we deduce that
\begin{equation*}
\int_{\mathbb{R}^3}|w_n(x)-w(x)|^2u^2(x)\,{\rm d}x\rightarrow0,
\end{equation*}
that is, \eqref{equ2-13} holds with $i=1$.

For $i=2$. By H\"{o}lder's inequality, we deduce that
\begin{align*}
\Big|B_n^{(2)}-&\int_{\mathbb{R}^3}\int_{\mathbb{R}^3}\frac{u^2(y)u(x)\varphi(x)}{|y-x|^{3-2s}}\,{\rm d}y\,{\rm d}x\Big|\leq\int_{\mathbb{R}^3}\int_{\mathbb{R}^3}\frac{|(u_n(y)u_n(x)-u(y))u(x))u(y)\varphi(x)|}{|y-x|^{3-2s}}\,{\rm d}x\,{\rm d}y\\
&\leq\|\varphi\|_{L^{\frac{3}{s}}}\int_{\mathbb{R}^3}\Big(\int_{\mathbb{R}^3}\frac{|u_n(y)u_n(x)-u(y)u(x)|^{\frac{3}{3-s}}}{|y-x|^{\frac{3(3-2s)}{3-s}}}\,{\rm d}x\Big)^{\frac{3-s}{3}}|u(y)|\,{\rm d}y\\
&\leq C\int_{\mathbb{R}^3}\Big(\int_{\mathbb{R}^3}\frac{|u_n(y)u_n(x)-u(y)u(x)|^{\frac{3}{3-s}}}{|y-x|^{\frac{3(3-2s)}{3-s}}}\,{\rm d}x\Big)^{\frac{3-s}{3}}|u(y)|\,{\rm d}y.
\end{align*}
Set
\begin{equation*}
\chi_n(y)=\int_{\mathbb{R}^3}\frac{|u_n(y)u_n(x)-u(y)u(x)|^{\frac{3}{3-s}}}{|y-x|^{\frac{3(3-2s)}{3-s}}}\,{\rm d}x.
\end{equation*}
Next we show that $w_n(y)\rightarrow0$ a.e. in $\mathbb{R}^3$.

Now we check
\begin{equation*}
\chi_n(y)\leq|u_n(y)|^{\frac{3}{3-s}}\int_{\mathbb{R}^3}\frac{|u_n(x)-u(x)|^{\frac{3}{3-s}}}{|y-x|^{\frac{3(3-2s)}{3-s}}}\,{\rm d}x+|u_n(y)-u(y)|^{\frac{3}{3-s}}\int_{\mathbb{R}^3}\frac{|u(x)|^{\frac{3}{3-s}}}{|y-x|^{\frac{3(3-2s)}{3-s}}}\,{\rm d}x,
\end{equation*}
and
\begin{align*}
\int_{\mathbb{R}^3}&\frac{|u_n(x)-u(x)|^{\frac{3}{3-s}}}{|y-x|^{\frac{3(3-2s)}{3-s}}}\,{\rm d}x
\leq\Big(\int_{|y-x|\geq R}|u_n(x)-u(x)|^{\frac{3}{3-2s}}\,{\rm d}x\Big)^{\frac{3-2s}{3-s}}\Big(\int_{|y-x|\geq R}\frac{1}{|y-x|^{\frac{3(3-2s)}{s}}}\,{\rm d}x\Big)^{\frac{s}{3-s}}\\
&+\Big(\int_{|y-x|< R}|u_n(x)-u(x)|^{\frac{3\alpha'}{3-s}}\,{\rm d}x\Big)^{\frac{1}{\alpha'}}\Big(\int_{|y-x|<R}\frac{1}{|y-x|^{\frac{3(3-2s)\alpha}{(3-s)}}}\,{\rm d}x\Big)^{\frac{1}{\alpha}},
\end{align*}
where $\frac{1}{\alpha}+\frac{1}{\alpha'}=1$ and $\alpha$ can be chosen by $\frac{2(3-s)}{3}<\alpha<\frac{3-s}{3-2s}$. Let $n\rightarrow\infty$ and then $R\rightarrow\infty$, we get $\int_{\mathbb{R}^3}\frac{|u_n(x)-u(x)|^{\frac{3}{3-s}}}{|y-x|^{\frac{3(3-2s)}{3-s}}}\,{\rm d}x\rightarrow0$, as $n\rightarrow\infty$. Since $u_n\rightarrow u$ a.e. $x\in\mathbb{R}^3$, this leads to $\chi_n(y)\rightarrow 0$ a.e. $x\in\mathbb{R}^3$.

Set
\begin{equation*}
\overline{\chi}_n(y)=\int_{\mathbb{R}^3}\frac{|u_n(x)|^{\frac{3}{3-s}}}{|y-x|^{\frac{3(3-2s)}{3-s}}}\,{\rm d}x,\quad\widetilde{\chi}(y)=\int_{\mathbb{R}^3}\frac{|u(x)|^{\frac{3}{3-s}}}{|y-x|^{\frac{3(3-2s)}{3-s}}}\,{\rm d}x.
\end{equation*}
By Hardy-Littlewood-Sobolev inequality (taking $r=\frac{3-s}{3-2s}$, $p=\frac{2(3-s)}{3}$ and $q=\frac{2(3-s)}{3-2s}$), we get
\begin{equation*}
\|\overline{\chi}_n\|_{\frac{2(3-s)}{3-2s}}\leq\||u_n|^{\frac{3}{3-s}}\|_{\frac{2(3-s)}{3}}=\|u_n\|_2^{\frac{3}{3-s}}\leq C.
\end{equation*}
Similarly, $\widetilde{\chi}\in L^{\frac{2(3-s)}{3-2s}}(\mathbb{R}^3)$. Thus
\begin{align*}
\int_{\mathbb{R}^3}|\chi_n^{\frac{3-s}{3}}(y)|^2\,{\rm d}y&\leq C\int_{\mathbb{R}^3}\Big(|u_n(y)|^2|\overline{\chi}_n(y)|^{\frac{2(3-s)}{3}}+(|u(y)|^2+|u_n(y)|^2)|\widetilde{\chi}(y)|^{\frac{2(3-s)}{3}}\Big)\,{\rm d}y\\
&\leq C\Big(\int_{\mathbb{R}^3}|u_n(y)|^{\frac{3}{s}}\,{\rm d}y\Big)^{\frac{2s}{3}}\Big(\int_{\mathbb{R}^3}|\overline{\chi}_n(y)|^{\frac{2(3-s)}{3-2s}}\,{\rm d}y\Big)^{\frac{3-2s}{3}}\\
&+C \Big[\Big(\int_{\mathbb{R}^3}|u_n(y)|^{\frac{3}{s}}\,{\rm d}y\Big)^{\frac{2s}{3}}+\Big(\int_{\mathbb{R}^3}|u(y)|^{\frac{3}{s}}\,{\rm d}y\Big)^{\frac{2s}{3}}\Big]\Big(\int_{\mathbb{R}^3}|\widetilde{\chi}(y)|^{\frac{2(3-s)}{3-2s}}\,{\rm d}y\Big)^{\frac{3-2s}{3}}\\
&\leq C.
\end{align*}
That is, $\chi_n^{\frac{3-s}{3}}\in L^2(\mathbb{R}^3)$ is bounded and by $\chi_n\rightarrow 0$ a.e. in $\mathbb{R}^3$, we have $\chi_n^{\frac{3-s}{3}}\rightharpoonup0$ in $L^2(\mathbb{R}^3)$. Since $u\in L^2(\mathbb{R}^3)$, $\int_{\mathbb{R}^3}\chi_n^{\frac{3-s}{3}}|u(y)|\,{\rm d}y\rightarrow0$.
\end{proof}

\subsection{Regularity for system \eqref{main}}

Let $u_{+}=\max\{u, 0\}$, $u^{-}=\min\{u,0\}$.

\begin{lemma}\label{lemr-2-1}
Let $u\in H^s(\mathbb{R}^N)$ be a solution of the equation
\begin{equation*}
(-\Delta)^s u=W(u_{+})\quad x\in\mathbb{R}^N
\end{equation*}
with $|W(u)|\leq C(|u|+|u|^p)$ for some $1\leq p\leq 2_s^{\ast}-1$ and $C>0$. Then $u\geq0$ for a.e. $x\in\mathbb{R}^N$.
\end{lemma}
\begin{proof}
Multiplying the above equation by $u_{-}=\min\{u,0\}$ and we integrate over $\mathbb{R}^N$, we obtain
\begin{equation*}
\int_{\mathbb{R}^N}(-\Delta)^su u_{-}\,{\rm d}x=\int_{\mathbb{R}^3}W(u_{+})u_{-}\,{\rm d}x=0.
\end{equation*}
Hence, by integrate by parts we get
\begin{equation*}
\int_{\mathbb{R}^N}\int_{\mathbb{R}^N}\frac{(u(x)-u(y))(u_{-}(x)-u_{-}(y))}{|x-y|^{N+2s}}\,{\rm d}x\,{\rm d}y=0.
\end{equation*}
Observe that
\begin{equation*}
(u(x)-u(y))(u_{-}(x)-u_{-}(y))\geq|u_{-}(x)-u_{-}(y)|^2.
\end{equation*}
In fact, the above inequality is trivial if $u(x)$ and $u(y)$ are both the same symbols. Therefore, we suppose that $u(x)\leq0$ and $u(y)\geq0$ (the symmetric situation is analogous). In this case
\begin{align*}
(u(x)-u(y))(u_{-}(x)-u_{-}(y))=(u(x)-u(y))u(x)&=u(x)^2-u(x)u(y)\\
&\geq|u(x)|^2=|u_{-}(x)-u_{-}(y)|^2.
\end{align*}
Hence
\begin{equation*}
\int_{\mathbb{R}^N}\int_{\mathbb{R}^N}\frac{|u_{-}(x)-u_{-}(y)|^2}{|x-y|^{N+2s}}\,{\rm d}x\,{\rm d}y=0
\end{equation*}
which implies that $u_{-}=0$ for a.e. $x\in\mathbb{R}^N$. Hence, $u(x)\geq0$ for a.e. $x\in\mathbb{R}^N$.
\end{proof}

\begin{lemma}\label{lemr-2-2}(\cite{DMV}, Proposition 4.4.1)
Let $u\in \mathcal{D}^{s,2}(\mathbb{R}^3)$ be a nonnegative solution to the problem
\begin{equation*}
(-\Delta)^s u=f(x,u)\quad \text{in}\,\, \mathbb{R}^N
\end{equation*}
and assume that $|f(x,u)|\leq C(1+|u|^{p})$, for some $1\leq p\leq 2_s^{\ast}-1$ and $C>0$. Then $u\in L^{\infty}(\mathbb{R}^N)$.
\end{lemma}
\begin{lemma}\label{lemr-2-3}(\cite{S}, Proposition 2.9)
Let $(-\Delta)^su=h(x)$. Assume that $u\in L^{\infty}(\mathbb{R}^N)$ and $h\in L^{\infty}(\mathbb{R}^N)$.\\
$(i)$ If $2s\leq1$, then $u\in C^{0,\alpha}(\mathbb{R}^N)$ for any $\alpha<2s$. Moreover,
\begin{equation*}
\|u\|_{C^{0,\alpha}}\leq C(\|u\|_{L^{\infty}}+\|h\|_{L^{\infty}})
\end{equation*}
for a constant $C$ depending only on $N$, $s$ and $\alpha$.\\
$(ii)$ If $2s>1$, then $u\in C^{1,\alpha}(\mathbb{R}^N)$ for any $\alpha<2s-1$. Moreover,
\begin{equation*}
\|u\|_{C^{1,\alpha}}\leq C(\|u\|_{L^{\infty}}+\|h\|_{L^{\infty}})
\end{equation*}
for a constant $C$ depending only on $N$, $s$ and $\alpha$.
\end{lemma}

If $(u,\phi)\in H^s(\mathbb{R}^3)\times\mathcal{D}^{s,2}(\mathbb{R}^3)$ is a nonnegative solution of problem \eqref{main}, then
\begin{align*}
\phi_u(x)&=\int_{\mathbb{R}^3}\frac{u^2(y)}{|x-y|^{3-2s}}\,{\rm d}y=\int_{|x-y|\geq1}\frac{u^2(y)}{|x-y|^{3-2s}}\,{\rm d}y+\int_{|x-y|<1}\frac{u^2(y)}{|x-y|^{3-2s}}\,{\rm d}y\\
&\leq\|u\|_{L^{2p'}(B_1(x))}\Big(\int_{|y-x|<1}\frac{1}{|y-x|^{p(3-2s)}}\,{\rm d}y\Big)^{\frac{1}{p}}\\
&+\|u\|_{L^{2q'}(B_1^c(x))}\Big(\int_{|y-x|\geq 1}\frac{1}{|y-x|^{q(3-2s)}}\,{\rm d}y\Big)^{\frac{1}{q}}\\
&\leq C\|u\|
\end{align*}
where $p(3-2s)<3$, $2p\in(2,2_s^{\ast})$ and $q(3-2s)>3$, $2q\in(2,2_s^{\ast})$, since $s\in(\frac{3}{4},1)$. The above inequality implies that $\phi\in L^{\infty}(\mathbb{R}^3)$. Hence, we rewrite the first equation in \eqref{main}
\begin{equation*}
(-\Delta)^s u=|u|^{p-1}u-V(x)u-\phi u:= g(x,u),
\end{equation*}
using Lemma \ref{lemr-2-1} and \ref{lemr-2-2} for the equation
\begin{equation*}
(-\Delta)^s u=g(x,u^{+}),
\end{equation*}
we obtain that $u^{+}\in L^{\infty}(\mathbb{R}^3)$. Similarly, we deduce that $u^{-}\in L^{\infty}(\mathbb{R}^3)$. By $(ii)$ of Lemma \ref{lemr-2-3}, for the two equations in system \eqref{main}, we see that $u\in C^{1,\alpha}(\mathbb{R}^3)$ and $\phi\in C^{1,\alpha}(\mathbb{R}^3)$ for any $\alpha<2s-1$. Therefore, $\partial_{x_i}u$ satisfies the equation
\begin{equation*}
(-\Delta)^s(\partial_{x_i}u)=\partial_{x_i}g(x,u),\quad x\in\mathbb{R}^3.
\end{equation*}

If $(x,\nabla V(x))\in L^{\infty}(\mathbb{R}^3)$, applying $(ii)$ of Lemma \ref{lemr-2-3} to $\partial_{x_i}u$ again, it follows that $\partial_{x_i}u\in C^{1,\alpha}(\mathbb{R}^3)$ for any $\alpha<2s-1$. Similarly, we can obtain that $\partial_{x_i}\phi\in C^{1,\alpha}(\mathbb{R}^3)$ for any $\alpha<2s-1$. Consequently, $u,\phi\in C^{2,\alpha}(\mathbb{R}^3)$ for any $\alpha<2s-1$.

If $(x,\nabla V(x))\in L^{\frac{2_s^{\ast}}{2_s^{\ast}-2}}(\mathbb{R}^3)$, from the above statement, we see that $u\in C^{1,\alpha}(\mathbb{R}^3)$ and by the same proof, we obtain that $\phi\in C^{2,\alpha}(\mathbb{R}^3)$ for any $\alpha<2s-1$.

\subsection{Pohozaev identity}
\begin{proposition}\label{pro2-1}
Assume that $(V_1)-(V_2)$ hold. Let $f\in C^1(\mathbb{R},\mathbb{R})$ satisfy that
\begin{equation*}
|f(t)|\leq C(|t|+|t|^p)\quad 1\leq p\leq 2_s^{\ast}-1, \,\, C>0
\end{equation*}
and $(u,\phi)\in H^s(\mathbb{R}^3)\times \mathcal{D}^{s,2}(\mathbb{R}^3)$ be a solution for problem
\begin{equation}\label{equ2-14}
\left\{
  \begin{array}{ll}
    (-\Delta)^su+V(x)u+\phi u=f(u), & \hbox{in $\mathbb{R}^3$,} \\
    (-\Delta)^s\phi=u^2,& \hbox{in $\mathbb{R}^3$.}
  \end{array}
\right.
\end{equation}
Then the Pohozaev identity holds true
\begin{align*}
\frac{3-2s}{2}\int_{\mathbb{R}^3}|(-\Delta)^{\frac{s}{2}}u|^2\,{\rm d}x&+\frac{3+2s}{4}\int_{\mathbb{R}^3}\phi u^2\,{\rm d}x+\frac{3}{2}\int_{\mathbb{R}^3}V(x)u^2\,{\rm d}x\\
&+\frac{1}{2}\int_{\mathbb{R}^3}u^2(x\cdot\nabla V(x))\,{\rm d}x=3\int_{\mathbb{R}^3}F(u)\,{\rm d}x,
\end{align*}
where $F(s)=\int_0^sf(t)\,{\rm d}t$.
\end{proposition}
\begin{proof}
From the former proof, we see that $u,\phi\in C^{2,\alpha}(\mathbb{R}^3)$ or $u\in C^{1,\alpha}(\mathbb{R}^3)$, $\phi\in C^{2,\alpha}(\mathbb{R}^3)$. Multiplying the first equation of \eqref{equ2-14} by $x\cdot\nabla u$, integrating on $B_R$ and using Proposition 1.6 in \cite{OS} and Lemma 3.1 in \cite{AM}, we obtain
\begin{align}\label{equ2-15}
\int_{B_R}(-\Delta)^su(x\cdot\nabla u)\,{\rm d}x&=\frac{2s-3}{2}\int_{B_R}u(-\Delta)^su\,{\rm d}x-\frac{\Gamma^2(1+s)}{2}\int_{\partial B_R}(\frac{u}{\delta^s})^2(x\cdot\nu)\,{\rm d}\sigma\nonumber\\
&=\frac{2s-3}{2}\int_{B_R}u(-\Delta)^su\,{\rm d}x-\frac{\Gamma^2(1+s)}{2R^{2s}}\int_{\partial B_R}u^2(x\cdot\nu)\,{\rm d}\sigma.
\end{align}
\begin{equation}\label{equ2-16}
\int_{B_R}\phi u(x\cdot\nabla u)\,{\rm d}x=-\frac{1}{2}\int_{B_R}u^2(x\cdot\nabla\phi)\,{\rm d}x-\frac{3}{2}\int_{B_R}\phi u^2\,{\rm d}x+\frac{R}{2}\int_{\partial B_R}\phi u^2\,{\rm d}\sigma.
\end{equation}
\begin{align}\label{equ2-17}
\int_{B_R}V(x) u(x\cdot\nabla u)\,{\rm d}x&=-\frac{1}{2}\int_{B_R}u^2(x\cdot\nabla V(x))\,{\rm d}x-\frac{3}{2}\int_{B_R}V(x) u^2\,{\rm d}x\\
&+\frac{R}{2}\int_{\partial B_R}V(x) u^2\,{\rm d}\sigma.
\end{align}
\begin{equation}\label{equ2-18}
\int_{B_R}f(u)(x\cdot\nabla u)\,{\rm d}x=-3\int_{B_R}F(u)\,{\rm d}x+R\int_{\partial B_R}F(u)\,{\rm d}\sigma.
\end{equation}
Multiplying the second equation of \eqref{equ2-14} by $(x\cdot\nabla \phi)$, and from proposition 1.6 in \cite{OS}, we get
\begin{align}\label{equ2-19}
\int_{B_R}u^2(x\cdot\nabla \phi)\,{\rm d}x=\int_{B_R}(-\Delta)^s\phi(x\cdot\nabla\phi)\,{\rm d}x&=\frac{2s-3}{2}\int_{B_R}\phi (-\Delta)^s\phi\,{\rm d}x\nonumber\\
&-\frac{\Gamma^2(1+s)}{2R^{2s}}\int_{\partial B_R}\phi^2(x\cdot\nu)\,{\rm d}\sigma.
\end{align}
It follows from \eqref{equ2-15},\eqref{equ2-16},\eqref{equ2-17} and \eqref{equ2-18}, that
\begin{align}\label{equ2-20}
&\frac{2s-3}{2}\int_{B_R}u(-\Delta)^su\,{\rm d}x-\frac{1}{2}\int_{B_R}u^2(x\cdot\nabla\phi)\,{\rm d}x-\frac{3}{2}\int_{B_R}\phi u^2\,{\rm d}x\nonumber\\
&-\frac{1}{2}\int_{B_R}u^2(x\cdot\nabla V(x))\,{\rm d}x-\frac{3}{2}\int_{B_R}V(x) u^2\,{\rm d}x+3\int_{B_R}F(u)\,{\rm d}x\nonumber\\
&=\frac{\Gamma^2(1+s)}{2R^{2s}}\int_{\partial B_R}u^2(x\cdot\nu)\,{\rm d}\sigma-\frac{R}{2}\int_{\partial B_R}\phi u^2\,{\rm d}\sigma-\frac{R}{2}\int_{\partial B_R}V(x) u^2\,{\rm d}\sigma\nonumber\\
&+R\int_{\partial B_R}F(u)\,{\rm d}\sigma.
\end{align}
By \eqref{equ2-19} and \eqref{equ2-20}, we deduce that
\begin{align}\label{equ2-21}
&\frac{2s-3}{2}\int_{B_R}u(-\Delta)^su\,{\rm d}x-\frac{2s-3}{4}\int_{B_R}\phi (-\Delta)^s\phi\,{\rm d}x-\frac{3}{2}\int_{B_R}\phi u^2\,{\rm d}x\nonumber\\
&-\frac{1}{2}\int_{B_R}u^2(x\cdot\nabla V(x))\,{\rm d}x-\frac{3}{2}\int_{B_R}V(x) u^2\,{\rm d}x+3\int_{B_R}F(u)\,{\rm d}x\nonumber\\
&=\frac{\Gamma^2(1+s)}{2R^{2s}}\int_{\partial B_R}u^2(x\cdot\nu)\,{\rm d}\sigma-\frac{R}{2}\int_{\partial B_R}\phi u^2\,{\rm d}\sigma-\frac{R}{2}\int_{\partial B_R}V(x) u^2\,{\rm d}\sigma\nonumber\\
&-\frac{\Gamma^2(1+s)}{2R^{2s}}\int_{\partial B_R}\phi^2(x\cdot\nu)\,{\rm d}\sigma+R\int_{\partial B_R}F(u)\,{\rm d}\sigma.
\end{align}
We can find a sequence $R_n\rightarrow+\infty$ such that the right side of \eqref{equ2-21} vanishes, hence we have proved that
\begin{align*}
&\frac{2s-3}{2}\int_{\mathbb{R}^3}u(-\Delta)^su\,{\rm d}x-\frac{2s-3}{4}\int_{\mathbb{R}^3}\phi (-\Delta)^s\phi\,{\rm d}x-\frac{3}{2}\int_{\mathbb{R}^3}\phi u^2\,{\rm d}x\\
&-\frac{1}{2}\int_{\mathbb{R}^3}u^2(x\cdot\nabla V(x))\,{\rm d}x-\frac{3}{2}\int_{\mathbb{R}^3}V(x) u^2\,{\rm d}x+3\int_{\mathbb{R}^3}F(u)\,{\rm d}x=0
\end{align*}
According to the fact
\begin{equation*}
\int_{\mathbb{R}^3}\phi(-\Delta)^s\phi\,{\rm d}x=\int_{\mathbb{R}^3}\phi u^2\,{\rm d}x,
\end{equation*}
therefore, we infer that
\begin{align*}
\frac{2s-3}{2}\int_{\mathbb{R}^3}u(-\Delta)^su\,{\rm d}x-&\frac{2s+3}{4}\int_{\mathbb{R}^3}\phi u^2\,{\rm d}x-\frac{1}{2}\int_{\mathbb{R}^3}u^2(x\cdot\nabla V(x))\,{\rm d}x\\
&-\frac{3}{2}\int_{\mathbb{R}^3}V(x) u^2\,{\rm d}x+3\int_{\mathbb{R}^3}F(u)\,{\rm d}x=0.
\end{align*}
The proof is completed.
\end{proof}

In the end of this section, we recall the well-known concentration-compactness principle of Lions \cite{Lions1,Lions2}.
\begin{proposition}\label{pro2-2}
Let $\rho_n(x)\in L^1(\mathbb{R}^N)$ be a non-negative sequence satifying
\begin{equation*}
\int_{\mathbb{R}^N}\rho_n(x)\,{\rm d}x=l>0.
\end{equation*}
Then there exists a subsequence, still denoted by $\{\rho_n(x)\}$ such that one of the following cases occurs.\\
(i) (compactness) there exists $y_n\in\mathbb{R}^N$, such that for any $\varepsilon>0$, exists $R>0$ such that
\begin{equation*}
\int_{B_R(y_n)}\rho_n(x)\,{\rm d}x\geq l-\varepsilon,\quad n=1,2,\cdots.
\end{equation*}
(ii) (vanishing) for any fixed $R>0$, there holds
\begin{equation*}
\lim_{n\rightarrow\infty}\sup_{y\in\mathbb{R}^N}\int_{B_R(y)}\rho_n(x)\,{\rm d}x=0.
\end{equation*}
(iii) (dichotomy) there exists $\alpha\in(0,l)$ such that for any $\varepsilon>0$, there exists $n_0\geq1$, $\rho_n^{(1)}$, $\rho_n^{(2)}\in L^1(\mathbb{R}^N)$, for $n\geq n_0$, there holds
\begin{equation*}
\|\rho_n-(\rho_n^{(1)}+\rho_n^{(2)})\|_{L^1(\mathbb{R}^N}<\varepsilon,\quad \Big|\int_{\mathbb{R}^N}\rho_n^{(1)}(x)\,{\rm d}x-\alpha\Big|<\varepsilon,\quad \Big|\int_{\mathbb{R}^N}\rho_n^{(2)}(x)\,{\rm d}x-(l-\alpha)\Big|<\varepsilon
\end{equation*}
and
\begin{equation*}
{\rm dist}({\rm supp}\rho_n^{(1)},{\rm supp}\rho_n^{(2)})\rightarrow\infty,\quad \text{as}\,\, n\rightarrow\infty.
\end{equation*}
\end{proposition}
The vanishing Lemma for fractional Sobolev space is stated as follows.
\begin{lemma}\label{lem2-3}(Vanishing Lemma, \cite{Sec})
Assume that $\{u_n\}$ is bounded in $H^{\alpha}(\mathbb{R}^N)$ and it satisfies
\begin{equation*}
\lim_{n\rightarrow+\infty}\sup_{y\in\mathbb{R}^N}\int_{B_R(y)}|u_n(x)|^2\,{\rm d}x=0
\end{equation*}
where $R>0$. Then $u_n\rightarrow0$ in $L^r(\mathbb{R}^N)$ for every $2<r<2_s^{\ast}$.
\end{lemma}

\section{The constant potential case}
In this section we will assume that $V$ is a positive constant. Without loss of generalization, we assume that $V\equiv1$. By Proposition \ref{pro2-1}, if $(u,\phi)\in H^s(\mathbb{R}^3)\times \mathcal{D}^{s,2}(\mathbb{R}^3)$ is a solution of \eqref{main}, then  it satisfies the Pohozaev identity
\begin{equation}\label{equ3-1}
\frac{3-2s}{2}\int_{\mathbb{R}^3}|(-\Delta)^{\frac{s}{2}}u|^2\,{\rm d}x+\frac{3}{2}\int_{\mathbb{R}^3}u^2\,{\rm d}x+\frac{3+2s}{4}\int_{\mathbb{R}^3}\phi_u^s u^2\,{\rm d}x=\frac{3}{p+1}\int_{\mathbb{R}^3}|u|^{p+1}\,{\rm d}x.
\end{equation}
For convenience, We denote
\begin{equation*}
\mathcal{P}(u)=\frac{3-2s}{2}\int_{\mathbb{R}^3}|(-\Delta)^{\frac{s}{2}}u|^2\,{\rm d}x+\frac{3}{2}\int_{\mathbb{R}^3}u^2\,{\rm d}x+\frac{3+2s}{4}\int_{\mathbb{R}^3}\phi_u^s u^2\,{\rm d}x-\frac{3}{p+1}\int_{\mathbb{R}^3}|u|^{p+1}\,{\rm d}x.
\end{equation*}
Set $u_{\theta}=\theta^{\alpha}u(\theta^{\beta}x)$, by computation, we deduce that
\begin{equation}\label{equ3-2}
\int_{\mathbb{R}^3}|(-\Delta)^{\frac{s}{2}}u_{\theta}|^2\,{\rm d}x=\theta^{2\alpha-3\beta+2\beta s}\int_{\mathbb{R}^3}|(-\Delta)^{\frac{s}{2}}u|^2\,{\rm d}x,\,\, \int_{\mathbb{R}^3}|u_{\theta}|^2\,{\rm d}x=\theta^{2\alpha-3\beta}\int_{\mathbb{R}^3}|u|^2\,{\rm d}x
\end{equation}
and
\begin{equation}\label{equ3-3}
\int_{\mathbb{R}^3}\phi_{u_{\theta}}^su_{\theta}^2\,{\rm d}x=\theta^{4\alpha-3\beta-2\beta s}\int_{\mathbb{R}^3}\phi_u^s u^2\,{\rm d}x,\,\, \int_{\mathbb{R}^3}|u_{\theta}|^{p+1}\,{\rm d}x=\theta^{\alpha(p+1)-3\beta}\int_{\mathbb{R}^3}|u|^{p+1}\,{\rm d}x.
\end{equation}
We take $\alpha=2s$ and $\beta=1$, then
\begin{align*}
\gamma(\theta)=I(u_{\theta})=&\frac{\theta^{6s-3}}{2}\int_{\mathbb{R}^3}|(-\Delta)^{\frac{s}{2}}u|^2\,{\rm d}x+\frac{\theta^{4s-3}}{2}\int_{\mathbb{R}^3}|u|^2\,{\rm d}x+\frac{\theta^{6s-3}}{4}\int_{\mathbb{R}^3}\phi_u^s u^2\,{\rm d}x\\
&-\frac{\theta^{2s(p+1)-3}}{p+1}\int_{\mathbb{R}^3}|u|^{p+1}\,{\rm d}x.
\end{align*}
If $p\in(2,2_s^{\ast}-1)$, we see that $I(u_{\theta})\rightarrow-\infty$ as $\theta\rightarrow+\infty$. We state this phenomenon as the following Lemma.
\begin{lemma}\label{lem3-1}
Let $p\in(2,2_s^{\ast}-1)$, then $I$ is not bounded from below.
\end{lemma}
By computation, we get
\begin{align*}
\gamma'(\theta)&=\frac{(6s-3)\theta^{6s-4}}{2}\int_{\mathbb{R}^3}|(-\Delta)^{\frac{s}{2}}u|^2\,{\rm d}x+\frac{(4s-3)\theta^{4s-4}}{2}\int_{\mathbb{R}^3}|u|^2\,{\rm d}x\\
&+\frac{(6s-3)\theta^{6s-4}}{4}\int_{\mathbb{R}^3}\phi_u^s u^2\,{\rm d}x-\frac{(2s(p+1)-3)\theta^{2s(p+1)-4}}{p+1}\int_{\mathbb{R}^3}|u|^{p+1}\,{\rm d}x.
\end{align*}
Therefore,
\begin{align*}
\gamma'(1)=&\frac{(6s-3)}{2}\int_{\mathbb{R}^3}|(-\Delta)^{\frac{s}{2}}u|^2\,{\rm d}x+\frac{4s-3}{2}\int_{\mathbb{R}^3}|u|^2\,{\rm d}x+\frac{(6s-3)}{4}\int_{\mathbb{R}^3}\phi_u^s u^2\,{\rm d}x\\
&-\frac{(2s(p+1)-3)}{p+1}\int_{\mathbb{R}^3}|u|^{p+1}\,{\rm d}x=2s\langle I'(u),u\rangle-\mathcal{P}(u).
\end{align*}
Define $\mathcal{G}:H^s(\mathbb{R}^3)\rightarrow\mathbb{R}$ as
\begin{equation*}
\mathcal{G}(u)=2s\langle I'(u),u\rangle-\mathcal{P}(u).
\end{equation*}
We shall study the functional $I$ on the manifold $\mathcal{M}$ defined as
\begin{equation*}
\mathcal{M}=\{u\in H^s(\mathbb{R}^3)\backslash\{0\}\,\,:\,\, \mathcal{G}(u)=0\}.
\end{equation*}
Clearly, if $u\in H^s(\mathbb{R}^3)$ is a nontrivial critical point of $I$, then $u\in\mathcal{M}$. Hence, if $(u,\phi)\in H^s(\mathbb{R}^3)\times \mathcal{D}^{s,2}(\mathbb{R}^3)$
is a solution of \eqref{main}, then $u\in \mathcal{M}$.

The following Lemma describes the properties of the manifold $\mathcal{M}$.
\begin{lemma}\label{lem3-2}
(1) For any $u\in H^s(\mathbb{R}^3)\backslash\{0\}$, there exists a unique number $\theta_0>0$ such that $u_{\theta}\in\mathcal{M}$. Moreover
\begin{equation*}
I(u_{\theta_0})=\max_{\theta\geq0}I(u_{\theta}).
\end{equation*}
(2) $0\not\in\partial\mathcal{M}$;\\
(3) $I(u)>0$ for all $u\in\mathcal{M}$;\\
(4) $\mathcal{G}'(u)\neq0$, for any $u\in\mathcal{M}$, that is, $\mathcal{M}$ is a $C^1$ manifold;\\
(5) $\mathcal{M}$ is a natural constraint of $I$, that is every critical point of $I|_{\mathcal{M}}$ is a critical point of $I$;\\
(6) There exists a positive constant $C>0$ such that $\|u\|_{p+1}\geq C$, for any $u\in\mathcal{M}$.\\
\end{lemma}

\begin{proof}
(1) We observe that
\begin{equation*}
u_{\theta}\in\mathcal{M}\,\,\Longleftrightarrow\,\,\theta\gamma'(\theta)=0\,\,\Longleftrightarrow\,\,\gamma'(\theta)=0\quad \text{for some}\,\, \theta>0.
\end{equation*}
Clearly, $\gamma(\theta)$ is positive for small $\theta$ and tends to $-\infty$ as $\theta\rightarrow+\infty$. Since $\gamma'$ is continuous, there exists at least one $\theta_0=\theta_0(u)>0$ such that $\gamma'(\theta)=0$, which means that $u_{\theta}\in\mathcal{M}$.

To show the uniqueness of $\theta_0$, note that $\gamma'(\theta_0)=0$ implies that
\begin{align}\label{equ3-4}
\frac{(6s-3)}{2}\int_{\mathbb{R}^3}&|(-\Delta)^{\frac{s}{2}}u|^2\,{\rm d}x+\frac{(6s-3)}{4}\int_{\mathbb{R}^3}\phi_u^s u^2\,{\rm d}x=\nonumber\\
&\frac{(2s(p+1)-3)\theta^{2s(p+1)-6s}}{p+1}\int_{\mathbb{R}^3}|u|^{p+1}\,{\rm d}x-\frac{(4s-3)\theta^{-2s}}{2}\int_{\mathbb{R}^3}|u|^2\,{\rm d}x\nonumber\\
&:=h(\theta).
\end{align}
Taking the derivative of $h(\theta)$, we have
\begin{align*}
h'(\theta)&=\frac{(2s(p+1)-6s)(2s(p+1)-3)\theta^{2s(p+1)-6s-1}}{p+1}\int_{\mathbb{R}^3}|u|^{p+1}\,{\rm d}x+s(4s-3)\\
&\theta^{-1-2s}\int_{\mathbb{R}^3}|u|^2\,{\rm d}x\\
&=\theta^{-1-2s}\Big(\frac{2s(p-2)(2s(p+1)-3)\theta^{2s(p-1)}}{p+1}\int_{\mathbb{R}^3}|u|^{p+1}\,{\rm d}x+s(4s-3)\\
&\int_{\mathbb{R}^3}|u|^2\,{\rm d}x\Big)>0.
\end{align*}
Therefore, $h(\theta)$ is an increasing function of $\theta$. As a consequence, there exists a unique $\theta_0>0$ such that \eqref{equ3-4} holds true. The uniqueness of $\theta_0$ is verified and (1) is proved.

(2) By the Sobolev embedding theorem, we have
\begin{align*}
\frac{(6s-3)}{2}\int_{\mathbb{R}^3}&|(-\Delta)^{\frac{s}{2}}u|^2\,{\rm d}x+\frac{(4s-3)}{2}\int_{\mathbb{R}^3}u^2\,{\rm d}x+\frac{(6s-3)}{4}\int_{\mathbb{R}^3}\phi_u^s u^2\,{\rm d}x\\
&-\frac{(2s(p+1)-3)}{p+1}\int_{\mathbb{R}^3}|u|^{p+1}\,{\rm d}x\\
&\geq\frac{(4s-3)}{2}\|u\|^2-C\frac{(2s(p+1)-3)}{p+1}\|u\|^{p+1}
\end{align*}
which is strictly positive for $\|u\|$ small. Then $0\not\in\partial\mathcal{M}$.

(3) For any $u\in\mathcal{M}$, let $k=I(u)$ and
\begin{equation*}
a=\frac{1}{2}\int_{\mathbb{R}^3}|(-\Delta)^{\frac{s}{2}}u|^2\,{\rm d}x, \,\, b=\frac{1}{2}\int_{\mathbb{R}^3}|u|^2\,{\rm d}x,\,\, c=\frac{1}{4}\int_{\mathbb{R}^3}\phi_u^s u^2\,{\rm d}x,
\end{equation*}
and
\begin{equation*}
d=\frac{1}{p+1}\int_{\mathbb{R}^3}|u|^{p+1}\,{\rm d}x.
\end{equation*}
Then $a,b,c,d$ are positive and satisfy the following identity
\begin{equation*}
\left\{
  \begin{array}{ll}
    a+b+c-d=k, & \hbox{} \\
    (6s-3)a+(4s-3)b+(6s-3)c-(2s(p+1)-3)d=0, & \hbox{}
  \end{array}
\right.
\end{equation*}
where the first equation comes from the definition of $I(u)$ and the second one is from the Pohozaev identity in Proposition \ref{pro2-1}. From the above relation, we can deduce that
\begin{equation*}
b=\frac{1}{2s}\Big(k(6s-3)-2s(p-2)d)\Big.
\end{equation*}
Since $b>0$ and $p>2$, we deduce that
\begin{equation*}
I(u)=k>\frac{2s(p-2)}{(6s-3)}d>0.
\end{equation*}

(4) By contradiction, suppose that $\mathcal{G}'(u)=0$ for some $u\in\mathcal{M}$. In a weak sense, the equation $\mathcal{G}'(u)=0$ can be written as
\begin{equation*}
\left\{
  \begin{array}{ll}
    (6s-3)(-\Delta)^su+(4s-3)u+(6s-3)\phi u=(2s(p+1)-3)|u|^{p-1}u, & \hbox{} \\
    (-\Delta)^s\phi=u^2. & \hbox{}
  \end{array}
\right.
\end{equation*}
Using the notations defined in (3), we have
\begin{equation*}
\left\{
  \begin{array}{ll}
    (6s-3)a+(4s-3)b+(6s-3)c-(2s(p+1)-3)d=0, & \hbox{} \\
    2(6s-3)a+2(4s-3)b+4(6s-3)c-(p+1)(2s(p+1)-3)d=0, & \hbox{} \\
    (3-2s)(6s-3)a+3(4s-3)b+(3+2s)(6s-3)c-3(2s(p+1)-3)d=0, & \hbox{}
  \end{array}
\right.
\end{equation*}
where the first equation is $u\in\mathcal{M}$, the second one is $\langle \mathcal{G}'(u),u\rangle=0$ and the last one comes from the Pohozaev identity in Proposition \ref{pro2-1}. From the above relation, we deduce that
\begin{equation*}
a=c=\frac{(p-1)(2s(p+1)-3)}{2(6s-3)}d,\,\, (4s-3)b+(p-2)((2s(p+1)-3))d=0,
\end{equation*}
which implies that $a=b=c=d=0$ since $s\in(\frac{3}{4},1)$, we achieve a contradiction with $a,b,c,d>0$. So $\mathcal{G}'(u)\neq0$ for every $u\in\mathcal{M}$ and by the implicit function Theorem, $\mathcal{M}$ is a $C^1$-manifold.

(5) Let $u$ be a critical point of the functional $I$, restricted to the manifold $\mathcal{M}$. By the theorem of Lagrange multipliers, there exists a $\mu\in\mathbb{R}$ such that
\begin{equation*}
I'(u)+\mu\mathcal{G}'(u)=0.
\end{equation*}
We will show that $\mu=0$. Evaluating the linear functional above at $u\in\mathcal{M}$, we obtain
\begin{align*}
\langle I'(u),u\rangle+\mu\langle\mathcal{G}'(u),u\rangle=0
\end{align*}
which is equivalent to
\begin{align*}
&\int_{\mathbb{R}^3}\Big(|(-\Delta)^{\frac{s}{2}} u|^2+|u|^2+\phi_u u^2-|u|^{p+1}\Big)\,{\rm d}x+\mu\Big((6s-3)\int_{\mathbb{R}^3}|(-\Delta)^{\frac{s}{2}}u|^2\,{\rm d}x\\
&+(4s-3)\int_{\mathbb{R}^3}|u|^2\,{\rm d}x+(6s-3)\int_{\mathbb{R}^3}\phi_u^s u^2\,{\rm d}x-(2s(p+1)-3)\int_{\mathbb{R}^3}|u|^{p+1}\,{\rm d}x\Big)\\
&=0.
\end{align*}
The above equality is associated with the systems
\begin{align*}
(-\Delta)^{s}u+u+\phi_uu-|u|^{p-1}u+\mu\Big((6s-3)&(-\Delta)^su+(4s-3)u+(6s-3)\phi_u^su\\
&-(2s(p+1)-3)|u|^{p-1}u\Big)=0
\end{align*}
which can be rewritten as
\begin{align*}
\Big(1+\mu(6s-3)\Big)(-\Delta)^{s}u+\Big(1+\mu(4s-3)&\Big)u+\Big(1+\mu(6s-3)\Big)\phi_u^su\\
&=\Big(1+\mu(2s(p+1)-3)\Big)|u|^{p-1}u.
\end{align*}
The solutions of this equation satisfy the following Pohozaev identity
\begin{align*}
&\frac{(3-2s)\Big(1+\mu(6s-3)\Big)}{2}\int_{\mathbb{R}^3}|(-\Delta)^{\frac{s}{2}}u|^2\,{\rm d}x+\frac{3\Big(1+\mu(4s-3)\Big)}{2}\int_{\mathbb{R}^3}|u|^2\,{\rm d}x\\
&+\frac{(3+2s)\Big(1+\mu(6s-3)\Big)}{4}\int_{\mathbb{R}^3}\phi_u^s u^2\,{\rm d}x=\frac{3\Big(1+\mu(2s(p+1)-3)\Big)}{p+1}\int_{\mathbb{R}^3}|u|^{p+1}\,{\rm d}x.
\end{align*}
Using the notations of (3), recalling that $u\in\mathcal{M}$, by multiplying the above equation by $u$ and integrating, and the Pohozaev identity for the above equation, we get the following linear systems of $a,b,c,d$. Namely
\begin{equation*}
\left\{
  \begin{array}{ll}
   a+b+c-d=k>0, & \hbox{} \\
    (6s-3)a+(4s-3)b+(6s-3)c-(2s(p+1)-3)d=0, & \hbox{}\\
    2(1+\mu(6s-3))a+2(1+\mu(4s-3))b+4(1+\mu(6s-3))c & \hbox{} \\
            -(p+1)(1+\mu(2s(p+1)-3))d=0, & \hbox{} \\
    (3-2s)(1+\mu(6s-3))a+3(1+\mu(4s-3))b+(3+2s)(1+\mu(6s-3))c & \hbox{}\\
-3(1+\mu(2s(p+1)-3))d=0, & \hbox{}
  \end{array}
\right.
\end{equation*}
By computation, the determinant of the coefficient matrix $A$ of the above systems is
\begin{equation*}
{\rm det}(A)=-16\mu s^3(1+\mu(6s-3)(p-1)(p-2).
\end{equation*}
Then
\begin{equation*}
{\rm det}(A)=0\,\,\Longleftrightarrow\,\,p=1,\,\,p=2,\,\,\mu=0,\,\,\mu=-\frac{1}{6s-3}.
\end{equation*}
We will show that $\mu$ must ne equal to zero by excluding the other two possibilities.

$(i)$ If $\mu\neq0$, $\mu\neq-\frac{1}{6s-3}$, the linear systems has a unique solution. Using the Cramer rule, we find the value of $d$:
\begin{equation*}
d=-\frac{3k(2s-1)(4s-3)}{4s^2(p-1)(p-2)}<0,
\end{equation*}
which contradicts with $d>0$.

$(ii)$ Assume that $\mu=-\frac{1}{6s-3}$. In such case, the third equation in the above linear systems changes into the following one
\begin{equation*}
2sb+s(p+1)(p-2)d=0
\end{equation*}
which is also impossible, since both $b$ and $d$ must be positive.

Therefore, $\mu=0$, and as a result, $I'(u)=0$, i.e., $u$ is a critical point of the functional $I$.

(6) For any $u\in\mathcal{M}$, $\mathcal{G}(u)=0$. By the Sobolev embedding inequality, we have
\begin{align*}
0&=\frac{(6s-3)}{2}\int_{\mathbb{R}^3}|(-\Delta)^{\frac{s}{2}}u|^2\,{\rm d}x+\frac{4s-3}{2}\int_{\mathbb{R}^3}|u|^2\,{\rm d}x+\frac{(6s-3)}{4}\int_{\mathbb{R}^3}\phi_u^s u^2\,{\rm d}x\\
&-\frac{(2s(p+1)-3)}{p+1}\int_{\mathbb{R}^3}|u|^{p+1}\,{\rm d}x\\
&\geq \frac{(4s-3)}{2}\|u\|^2-\frac{(2s(p+1)-3)}{p+1}\int_{\mathbb{R}^3}|u|^{p+1}\,{\rm d}x\\
&\geq \frac{(4s-3)}{2}C\|u\|_{p+1}^2-\frac{(2s(p+1)-3)}{p+1}\int_{\mathbb{R}^3}|u|^{p+1}\,{\rm d}x.
\end{align*}
Then
\begin{equation*}
\|u\|_{p+1}\geq \Big(\frac{C(4s-3)(p+1)}{2(2s(p+1)-3)}\Big)^{\frac{1}{p-1}}.
\end{equation*}
\end{proof}
\begin{remark}\label{rem3-1}
The map $H^s(\mathbb{R}^3\backslash\{0\})\rightarrow (0,+\infty): u\rightarrow \theta(u)$ is continuous.\\
In fact, assume $u_n\rightarrow u$ in $H^s(\mathbb{R}^3)$, it is easy to show that $\theta(u_n)$ is bounded and as a result, we may assume that $\theta(u_n)\rightarrow \theta_1$. By the uniqueness of $\theta(u)$, we can obtain that $\theta_1=\theta(u)$.
\end{remark}
By (5) of Lemma \ref{lem3-2}, we can find critical points of $I$ restricted to $\mathcal{M}$. Set
\begin{equation*}
c_1=\inf_{g\in\Gamma}\max_{\theta\in[0,1]}I(g(\theta)),\,\,c_2=\inf_{u\neq0}\max_{\theta\geq0}I(u_{\theta}),\,\, c_3=\inf_{u\in\mathcal{M}}I(u),
\end{equation*}
where
\begin{equation*}
\Gamma=\{g\in C([0,1],H^s(\mathbb{R}^3))\,\,|\,\, g(0)=0,\,\, I(g(1))\leq0, \,\, g(1)\neq0\}.
\end{equation*}

\begin{lemma}\label{lem3-3}
The following equalities hold
\begin{equation*}
c:=c_1=c_2=c_3>0.
\end{equation*}
\end{lemma}
\begin{proof}
The proof is similar to that of Proposition 3.11 in \cite{CS}. We give a detailed proof here for readers; convenience. Lemma \ref{lem3-1} implies that $c_2=c_3$.

From Lemma \ref{lem3-1}, we see that $I(u_{\theta})<0$ for $u\in H^s(\mathbb{R}^3)\backslash\{0\}$ and $\theta$ large enough, we obtain that $c_1\leq c_2$.

On the other hand, for any $\gamma\in \Gamma$, we claim that $\gamma([0,1])\cap\mathcal{M}\neq\emptyset$. In fact, for every $u\in\{u\in H^s(\mathbb{R}^3)\backslash\{0\}\,\,|\,\,\mathcal{G}(u)\geq0\}\cup\{0\}$, by the Sobolev embedding theorem, we have
\begin{align*}
&\frac{(6s-3)}{2}\int_{\mathbb{R}^3}|(-\Delta)^{\frac{s}{2}}u|^2\,{\rm d}x+\frac{4s-3}{2}\int_{\mathbb{R}^3}|u|^2\,{\rm d}x+\frac{(6s-3)}{4}\int_{\mathbb{R}^3}\phi_u^s u^2\,{\rm d}x-\frac{(2s(p+1)-3)}{p+1}\int_{\mathbb{R}^3}|u|^{p+1}\,{\rm d}x\\
&\geq\frac{4s-3}{2}\|u\|^2-C\frac{(2s(p+1)-3)}{p+1}\|u\|^{p+1}
\end{align*}
which implies that there exists a small neighberhood of $0$ such that it is contained in $\{u\in H^s(\mathbb{R}^3)\backslash\{0\}\,\,|\,\,\mathcal{G}(u)\geq0\}\cup\{0\}$. And also for every $u\in\{u\in H^s(\mathbb{R}^3)\backslash\{0\}\,\,|\,\,\mathcal{G}(u)\geq0\}\cup\{0\}$, we have
\begin{align*}
(6s-3)I(u)=\mathcal{G}(u)+s\int_{\mathbb{R}^3}u^2\,{\rm d}x+\frac{2s(p-2)}{p+1}\int_{\mathbb{R}^3}|u|^{p+1}\,{\rm d}x\geq0
\end{align*}
and $I(u)>0$ if $u\neq0$. Hence, for any $\gamma\in\Gamma$ satisfying $\gamma(0)=0$, $I(\gamma(1))\leq0$ and $\gamma(1)\neq0$, the carve $\gamma(\theta)$ must across the manifold $\mathcal{M}$. Therefore, $c_1\geq c_3$.
\end{proof}
\begin{remark}\label{rem3-2}
Let
\begin{equation*}
\widetilde{c_2}=\inf_{u\neq0}\max_{t\geq0}I(tu),\quad \widetilde{c_3}=\inf_{u\in\mathcal{N}}I(u),
\end{equation*}
where $\mathcal{N}$ is the Nehari manifold defined as
\begin{align*}
\mathcal{N}&=\{u\in H^s(\mathbb{R}^3)\backslash\{0\}\,\,\Big|\,\, \int_{\mathbb{R}^3}|(-\Delta)^{\frac{s}{2}}u|^2\,{\rm d}x+\int_{\mathbb{R}^3}|u|^2\,{\rm d}x+\int_{\mathbb{R}^3}\phi_u^su^2\,{\rm d}x=\int_{\mathbb{R}^3}|u|^{p+1}\,{\rm d}x\}.
\end{align*}
Therefore
\begin{equation*}
c=c_1=c_2=c_3=\widetilde{c_2}=\widetilde{c_3}.
\end{equation*}
\end{remark}
\begin{lemma}\label{lem3-4}
Let $\{u_n\}\subset\mathcal{M}$ be a minimizing sequence for $c$ which is give by Lemma \ref{lem3-3}. Then there exists $\{y_n\}\subset\mathbb{R}^3$ such that for any $\varepsilon>0$, there exists an $R>0$ satisfying
\begin{equation*}
\int_{\mathbb{R}^3\backslash B_R(y_n)}\Big(\int_{\mathbb{R}^3}\frac{|u_n(x)-u_n(y)|^2}{|x-y|^{3+2s}}\,{\rm d}y+u_n^2\Big)\,{\rm d}x\leq \varepsilon.
\end{equation*}
\end{lemma}
\begin{proof}
Let $\{u_n\}\subset\mathcal{M}$ such that
\begin{equation}\label{equ3-5}
\lim_{n\rightarrow\infty}I(u_n)=c>0.
\end{equation}
Since $\{u_n\}\subset\mathcal{M}$, we see that
\begin{align}\label{equ3-6}
I(u_n)&=\frac{1}{2}\int_{\mathbb{R}^3}|(-\Delta)^{\frac{s}{2}}u_n|^2\,{\rm d}x+\frac{1}{2}\int_{\mathbb{R}^3}|u_n|^2\,{\rm d}x+\frac{1}{4}\int_{\mathbb{R}^3}\phi_{u_n}^su_n^2\,{\rm d}x-\frac{1}{p+1}\int_{\mathbb{R}^3}|u_n|^{p+1}\,{\rm d}x\nonumber\\
&=(\frac{1}{2}-\frac{(6s-3)}{2(2s(p+1)-3)})\int_{\mathbb{R}^3}|(-\Delta)^{\frac{s}{2}}u_n|^2\,{\rm d}x+(\frac{1}{2}-\frac{(4s-3)}{2(2s(p+1)-3)})\nonumber\\
&\int_{\mathbb{R}^3}|u_n|^2\,{\rm d}x+(\frac{1}{4}-\frac{(6s-3)}{4(2s(p+1)-3)})\int_{\mathbb{R}^3}\phi_{u_n}^su_n^2\,{\rm d}x\nonumber\\
&=\frac{s(p-2)}{(2s(p+1)-3)}\int_{\mathbb{R}^3}|(-\Delta)^{\frac{s}{2}}u_n|^2\,{\rm d}x+\frac{s(p-1)}{(2s(p+1)-3)})\int_{\mathbb{R}^3}|u_n|^2\,{\rm d}x\nonumber\\
&+\frac{s(p-2)}{2(2s(p+1)-3)})\int_{\mathbb{R}^3}\phi_{u_n}^su_n^2\,{\rm d}x\\
&\:= J(u_n)\geq0.
\end{align}
From \eqref{equ3-5}, it follows that $\{u_n\}$ is bounded in $H^s(\mathbb{R}^3)$.

Next, we use Proposition \ref{pro2-2} to conclude the compactness of the sequence $\{u_n\}$. Let
\begin{align*}
\rho_n=&\frac{s(p-2)}{C(s)(2s(p+1)-3)}\int_{\mathbb{R}^3}\frac{|u_n(x)-u_n(y)|^2}{|x-y|^{3+2s}}\,{\rm d}y+\frac{s(p-1)}{(2s(p+1)-3)})|u_n|^2+\frac{s(p-2)}{2(2s(p+1)-3)})\phi_{u_n}^su_n^2,
\end{align*}
then by \eqref{equ3-6}, $\{\rho_n\}$ is a sequence of nonnegative $L^1$ functions on $\mathbb{R}^3$ and by \eqref{equ3-5}, it satisfies
\begin{equation}\label{equ3-8}
 \int_{\mathbb{R}^3}\rho_n\,{\rm d}x\rightarrow c.
\end{equation}
$(i)$ Vanishing does not occur. Suppose by contradiction, then for all $R>0$,
\begin{equation*}
\lim_n\sup_{y\in\mathbb{R}^3}\int_{B_R(y)}\rho_n(x)\,{\rm d}x=0.
\end{equation*}
Hence
\begin{equation*}
\lim_n\sup_{y\in\mathbb{R}^3}\int_{B_R(y)}|u_n(x)|^2\,{\rm d}x=0.
\end{equation*}
By vanishing Lemma \ref{lem2-3}, we have that $u_n\rightarrow 0$ in $L^t(\mathbb{R}^3)$ for $2<t<2_s^{\ast}$. As a sequence, from \eqref{equ2-6}, it follows that
\begin{equation*}
\int_{\mathbb{R}^3}\phi_{u_n}^su_n^2\,{\rm d}x\rightarrow0.
\end{equation*}
Since $\{u_n\}\subset\mathcal{M}$, it is easy to verify that
\begin{equation*}
\lim_{n\rightarrow\infty}\Big(\frac{(6s-3)}{2}\int_{\mathbb{R}^3}|(-\Delta)^{\frac{s}{2}}u_n|^2\,{\rm d}x+\frac{(4s-3)}{2}\int_{\mathbb{R}^3}|u_n|^2\,{\rm d}x\Big)=0.
\end{equation*}
Therefore, we obtain
\begin{align*}
\lim_{n\rightarrow\infty}I(u_n)=0,
\end{align*}
which contradicts with \eqref{equ3-5}.

$(ii)$ Dichotomy does not occur.

Suppose by contradiction that there exists an $\alpha\in(0,c)$ and $\{y_n\}\subset\mathbb{R}^3$ such that for all $\varepsilon>0$, there exists $\{R_n\}\subset\mathbb{R}^{+}$ with $R_n\rightarrow+\infty$ satisfying
\begin{equation}\label{equ3-9}
\limsup_{n\rightarrow\infty}\Big(\Big|\alpha-\int_{B_{R_n}(y_n)}\rho_n\,{\rm d}x\Big|+\Big|c-\alpha-\int_{\mathbb{R}^3\backslash B_{2R_n}(y_n)}\rho_n\,{\rm d}x\Big|\Big)<\varepsilon.
\end{equation}
Let $\xi:\mathbb{R}^{+}\cup\{0\}\rightarrow\mathbb{R}^{+}$ be a cut-off function such that $0\leq\xi\leq1$, $\xi(t)=1$ for $t\leq1$, $\xi(t)=0$ for $t\geq2$ and $|\xi'(t)|\leq2$. Set
\begin{equation*}
v_n(x)=\xi\Big(\frac{|x-y_n|}{R_n}\Big)u_n(x),\quad w_n(x)=\Big(1-\xi\Big(\frac{|x-y_n|}{R_n}\Big)\Big)u_n(x),
\end{equation*}
clearly, there hold
\begin{align}\label{equ3-10}
&(v_n(x)-v_n(y))(w_n(x)-w_n(y))=\nonumber\\
&\left\{
                                 \begin{array}{ll}
                                   0, & \hbox{$(x,y)\in B_{R_n}(y_n)\times B_{R_n}(y_n)$,} \\
                                   -u_n(x)u_n(y), & \hbox{$(x,y)\in B_{R_n}(y_n)\times B_{2R_n}^c(y_n)$,} \\
                                   -(u_n(x)-v_n(y))w_n(y), & \hbox{$(x,y)\in B_{R_n}(y_n)\times B_{2R_n}(y_n)\backslash B_{R_n}(y_n)$,} \\
                                   (v_n(x)-u_n(y))w_n(x), & \hbox{$(x,y)\in B_{2R_n}(y_n)\backslash B_{R_n}(y_n)\times B_{R_n}(y_n)$,} \\
                                  (v_n(x)-v_n(y))(w_n(x)-w_n(y)), & \hbox{$(x,y)\in B_{2R_n}(y_n)\backslash B_{R_n}(y_n)\times B_{2R_n}(y_n)\backslash B_{R_n}(y_n)$,} \\
                                   v_n(x)(w_n(x)-u_n(y)), & \hbox{$(x,y)\in B_{2R_n}(y_n)\backslash B_{R_n}(y_n)\times B_{2R_n}^c(y_n)$,}\\
                                   -u_n(y)u_n(x), & \hbox{$(x,y)\in B_{2R_n}^c(y_n)\times B_{R_n}(y_n)$,}\\
                                    -v_n(y)(u_n(x)-w_n(y)), & \hbox{$(x,y)\in B_{2R_n}^c(y_n)\times B_{2R_n}(y_n)\backslash B_{R_n}(y_n)$,}\\
                                    0, & \hbox{$(x,y)\in B_{2R_n}^c(y_n)\times B_{2R_n}^c(y_n)$,}\\
                                 \end{array}
                               \right.
\end{align}
Then by \eqref{equ3-9}, we see that
\begin{align*}
\liminf_{n\rightarrow\infty}\int_{\mathbb{R}^3}J(v_n)\,{\rm d}x&\geq\liminf_{n\rightarrow\infty}\int_{B_{R_n}(y_n)}J(v_n)\,{\rm d}x=\liminf_{n\rightarrow\infty}\int_{B_{R_n}(y_n)}J(u_n)\,{\rm d}x
&=\liminf_{n\rightarrow\infty}\int_{B_{R_n}(y_n)}\rho_n\,{\rm d}x\geq\alpha
\end{align*}
and
\begin{align*}
\liminf_{n\rightarrow\infty}\int_{\mathbb{R}^3}J(w_n)\,{\rm d}x&\geq\liminf_{n\rightarrow\infty}\int_{\mathbb{R}^3\backslash B_{2R_n}(y_n)}J(w_n)\,{\rm d}x=\liminf_{n\rightarrow\infty}\int_{\mathbb{R}^3\backslash B_{2R_n}(y_n)}J(u_n)\,{\rm d}x\\
&=\liminf_{n\rightarrow\infty}\int_{\mathbb{R}^3\backslash B_{2R_n}(y_n)}\rho_n\,{\rm d}x\geq c-\alpha.
\end{align*}
Denote $\Omega_n=B_{2R_n}(y_n)\backslash B_{R_n}(y_n)$, by \eqref{equ3-8} and the above two inequalities, we get
\begin{align*}
0\leq\lim_{n\rightarrow\infty}\int_{\Omega_n}\rho_n\,{\rm d}x&=\lim_{n\rightarrow\infty}\Big(\int_{\mathbb{R}^3}\rho_n\,{\rm d}x-\int_{B_{R_n}(y_n)}\rho_n\,{\rm d}x-\int_{\mathbb{R}^3\backslash B_{2R_n}(y_n)}\rho_n\,{\rm d}x\Big)\\
&\leq c-\liminf_{n\rightarrow\infty}\int_{B_{R_n}(y_n)}\rho_n\,{\rm d}x-\liminf_{n\rightarrow\infty}\int_{\mathbb{R}^3\backslash B_{2R_n}(y_n)}\rho_n\,{\rm d}x\\
&\leq c-\alpha-(c-\alpha)=0,
\end{align*}
that is
\begin{equation*}
\lim_{n\rightarrow\infty}\int_{\Omega_n}\rho_n\,{\rm d}x=0
\end{equation*}
which means that
\begin{equation}\label{equ3-11}
\int_{\Omega_n}\Big(\int_{\mathbb{R}^3}\frac{|u_n(x)-u_n(y)|^2}{|x-y|^{3+2s}}\,{\rm d}y+u_n^2\Big)\,{\rm d}x\,{\rm d}x\rightarrow0\,\,\text{and}\,\, \int_{\Omega_n}\phi_{u_n}^su_n^2\,{\rm d}x\rightarrow0
\end{equation}
as $n\rightarrow\infty$. Hence, by Lemma \ref{lemp-2-2}, we have
\begin{align}\label{equ3-12}
\int_{\Omega_n}\int_{\mathbb{R}^3}\frac{|v_n(x)-v_n(y)|^2}{|x-y|^{3+2s}}\,{\rm d}y\leq C\int_{\Omega_n}|u_n|^2\,{\rm d}x+2\int_{\Omega_n}\int_{\mathbb{R}^3}\frac{|u_n(x)-u_n(y)|^2}{|x-y|^{3+2s}}\,{\rm d}y\rightarrow0
\end{align}
and
\begin{align}\label{equ3-13}
\int_{\Omega_n}\int_{\mathbb{R}^3}\frac{|w_n(x)-w_n(y)|^2}{|x-y|^{3+2s}}\,{\rm d}y\leq C\int_{\Omega_n}|u_n|^2\,{\rm d}x+2\int_{\Omega_n}\int_{\mathbb{R}^3}\frac{|u_n(x)-u_n(y)|^2}{|x-y|^{3+2s}}\,{\rm d}y\rightarrow0.
\end{align}
Now we will show that
\begin{equation}\label{equ3-14}
\int_{\mathbb{R}^3}|(-\Delta)^{\frac{s}{2}}u_n|^2\,{\rm d}x=\int_{\mathbb{R}^3}|(-\Delta)^{\frac{s}{2}}v_n|^2\,{\rm d}x+\int_{\mathbb{R}^3}|(-\Delta)^{\frac{s}{2}}w_n|^2\,{\rm d}x+o_n(1),
\end{equation}
\begin{equation}\label{equ3-15}
\int_{\mathbb{R}^3}u_n^2(x)\,{\rm d}x=\int_{\mathbb{R}^3}v_n^2(x)\,{\rm d}x+\int_{\mathbb{R}^3}w_n^2(x)\,{\rm d}x+o_n(1),
\end{equation}
\begin{equation}\label{equ3-16}
\int_{\mathbb{R}^3}|u_n|^{p+1}\,{\rm d}x=\int_{\mathbb{R}^3}|v_n|^{p+1}\,{\rm d}x+\int_{\mathbb{R}^3}|w_n|^{p+1}\,{\rm d}x+o_n(1),
\end{equation}
and
\begin{equation}\label{equ3-17}
\int_{\mathbb{R}^3}\phi_{u_n}^su_n^2\,{\rm d}x\geq\int_{\mathbb{R}^3}\phi_{v_n}^sv_n^2\,{\rm d}x+\int_{\mathbb{R}^3}\phi_{w_n}^sw_n^2\,{\rm d}x+o_n(1).
\end{equation}

{\bf Proof of \eqref{equ3-14}.}
Observe that
\begin{align*}
&C(s)\int_{\mathbb{R}^3}|(-\Delta)^{\frac{s}{2}}u_n|^2\,{\rm d}x=\int_{\mathbb{R}^3}\int_{\mathbb{R}^3}\frac{|u_n(x)-u_n(y)|^2}{|x-y|^{3+2s}}\,{\rm d}y\,{\rm d}x\\
&=\int_{\mathbb{R}^3}\int_{\mathbb{R}^3}\frac{|v_n(x)-v_n(y)|^2}{|x-y|^{3+2s}}\,{\rm d}y\,{\rm d}x+\int_{\mathbb{R}^3}\int_{\mathbb{R}^3}\frac{|w_n(x)-w_n(y)|^2}{|x-y|^{3+2s}}\,{\rm d}y\,{\rm d}x\\
&+2\int_{\mathbb{R}^3}\int_{\mathbb{R}^3}\frac{(v_n(x)-v_n(y))(w_n(x)-w_n(y))}{|x-y|^{3+2s}}\,{\rm d}y\,{\rm d}x.
\end{align*}
From \eqref{equ3-10}, we split the last integral as follows.
\begin{align*}
\int_{\mathbb{R}^3}\int_{\mathbb{R}^3}\frac{(v_n(x)-v_n(y))(w_n(x)-w_n(y))}{|x-y|^{3+2s}}\,{\rm d}y\,{\rm d}x=\sum_{i=1}^4I_i
\end{align*}
where
\begin{equation*}
I_1=-2\int_{B_{R_n}(y_n)\times B_{2R_n}^c(y_n)}\frac{u_n(x)u_n(y)}{|x-y|^{3+2s}}\,{\rm d}y\,{\rm d}x
\end{equation*}
\begin{equation*}
I_2=2\int_{B_{R_n}(y_n)\times B_{2R_n}(y_n)\backslash B_{R_n}(y_n)}\frac{(v_n(x)-v_n(y))(w_n(x)-w_n(y))}{|x-y|^{3+2s}}\,{\rm d}y\,{\rm d}x
\end{equation*}
\begin{equation*}
I_3=\int_{B_{2R_n}(y_n)\backslash B_{R_n}(y_n)\times B_{2R_n}(y_n)\backslash B_{R_n}(y_n)}\frac{(v_n(x)-v_n(y))(w_n(x)-w_n(y))}{|x-y|^{3+2s}}\,{\rm d}y\,{\rm d}x
\end{equation*}
\begin{equation*}
I_4=2\int_{B_{2R_n}(y_n)\backslash B_{R_n}(y_n)\times B_{2R_n}^c(y_n)}\frac{(v_n(x)-v_n(y))(w_n(x)-u_n(y))}{|x-y|^{3+2s}}\,{\rm d}y\,{\rm d}x.
\end{equation*}
We estimate $I_i$ as follows.

{\bf Estimating $I_1$}
\begin{align*}
|I_1|&\leq\int_{B_{R_n}(y_n)\times B_{2R_n}^c(y_n)}\frac{u_n^2(x)+u_n^2(y)}{|x-y|^{3+2s}}\,{\rm d}x\,{\rm d}y=\int_{B_{R_n}(y_n)}u_n^2(x)\,{\rm d}x\int_{B_{2R_n}^c(y_n)}\frac{1}{|x-y|^{3+2s}}\,{\rm d}y\\
&+\int_{B_{2R_n}^c(y_n)}u_n^2(y)\Big(\int_{B_{R_n}(y_n)}\frac{1}{|x-y|^{3+2s}}\,{\rm d}x\Big)\,{\rm d}y\\
&\leq\int_{B_{R_n}(y_n)}u_n^2(x)\,{\rm d}x\int_{|x-y|\geq R_n}\frac{1}{|x-y|^{3+2s}}\,{\rm d}y+\int_{B_{2R_n}^c(y_n)}u_n^2(y)\,{\rm d}y\int_{|x-y|\geq R_n}\frac{1}{|x-y|^{3+2s}}\,{\rm d}x\\
&\leq\frac{C}{R_n^{2s}}\rightarrow0,
\end{align*}
where we have used the fact that $x\in B_{R_n}(y_n)$, $y\in B_{2R_n}^c(y_n)$ implies that $|x-y_n|\geq R_n$.

{\bf Estimating $I_2$}
\begin{align*}
|I_2|&\leq\int_{B_{R_n}(y_n)\times B_{2R_n}(y_n)\backslash B_{R_n}(y_n)}\frac{|v_n(x)-v_n(y)|^2}{|x-y|^{3+2s}}\,{\rm d}x\,{\rm d}y\\
&+\int_{B_{R_n}(y_n)\times B_{2R_n}(y_n)\backslash B_{R_n}(y_n)}\frac{|w_n(x)-w_n(y)|^2}{|x-y|^{3+2s}}\,{\rm d}x\,{\rm d}y\\
&\leq C\Big(\int_{\Omega_n}|u_n(x)|^2\,{\rm d}x+\int_{B_{R_n}(y_n)\times B_{2R_n}(y_n)\backslash B_{R_n}(y_n)}\frac{|u_n(x)-u_n(y)|^2}{|x-y|^{3+2s}}\,{\rm d}x\,{\rm d}y\Big)\\
&\leq C\Big(\int_{\Omega_n}|u_n(x)|^2\,{\rm d}x+\int_{\Omega_n\times\mathbb{R}^3}\frac{|u_n(x)-u_n(y)|^2}{|x-y|^{3+2s}}\,{\rm d}x\,{\rm d}y\Big)\\
&\rightarrow0,
\end{align*}
where we have used Lemma \ref{lemp-2-2} and \eqref{equ3-11}.

{\bf Estimating $I_3$}
\begin{align*}
|I_3|&\leq\int_{\Omega_n\times \Omega_n}\frac{|v_n(x)-v_n(y)|^2}{|x-y|^{3+2s}}\,{\rm d}x\,{\rm d}y+\int_{\Omega_n\times \Omega_n}\frac{|w_n(x)-w_n(y)|^2}{|x-y|^{3+2s}}\,{\rm d}x\,{\rm d}y\\
&\leq\int_{\Omega_n\times \mathbb{R}^3}\frac{|v_n(x)-v_n(y)|^2}{|x-y|^{3+2s}}\,{\rm d}x\,{\rm d}y+\int_{\Omega_n\times \mathbb{R}^3}\frac{|w_n(x)-w_n(y)|^2}{|x-y|^{3+2s}}\,{\rm d}x\,{\rm d}y\\
&\rightarrow0,
\end{align*}
where we have used \eqref{equ3-12} and \eqref{equ3-13}.

Similar to the proof of $I_2$, we can deduce that $I_4\rightarrow0$ as $n\rightarrow\infty$. Therefore, we have proved that
\begin{equation*}
\int_{\mathbb{R}^3}\int_{\mathbb{R}^3}\frac{(v_n(x)-v_n(y))(w_n(x)-w_n(y))}{|x-y|^{3+2s}}\,{\rm d}y\,{\rm d}x=o_n(1),
\end{equation*}
and thus \eqref{equ3-14} holds true.

{\bf Proof of \eqref{equ3-15}.}
Using \eqref{equ3-11}, we infer that
\begin{align*}
\int_{\mathbb{R}^3}|u_n|^2\,{\rm d}x-&\int_{\mathbb{R}^3}|v_n|^2\,{\rm d}x-\int_{\mathbb{R}^3}|w_n|^2\,{\rm d}x=\int_{\mathbb{R}^3}\Big(1-\xi^2-(1-\xi)^2\Big)|u_n|^2\,{\rm d}x\\
&=\int_{\Omega_n}\Big(1-\xi^2-(1-\xi)^2\Big)|u_n|^2\,{\rm d}x\leq\int_{\Omega_n}|u_n|^2\,{\rm d}x\rightarrow0,
\end{align*}
and thus \eqref{equ3-15} is true.

{\bf Proof of \eqref{equ3-16}.}
Using \eqref{equ3-11} and H\"{o}lder's inequality, we have
\begin{align*}
\int_{\mathbb{R}^3}|u_n|^{p+1}\,{\rm d}x-&\int_{\mathbb{R}^3}|v_n|^{p+1}\,{\rm d}x-\int_{\mathbb{R}^3}|w_n|^{p+1}\,{\rm d}x=\int_{\mathbb{R}^3}\Big(1-\xi^{p+1}-(1-\xi)^{p+1}\Big)\\
&|u_n|^{p+1}\,{\rm d}x=\int_{\Omega_n}\Big(1-\xi^{p+1}-(1-\xi)^{p+1}\Big)|u_n|^{p+1}\,{\rm d}x\\
&\leq\int_{\Omega_n}|u_n|^{p+1}\,{\rm d}x\\
&\leq\Big(\int_{\Omega_n}|u_n|^2\,{\rm d}x\Big)^{\frac{(1-\theta)(p+1)}{2}}\Big(\int_{\mathbb{R}^3}|u_n|^{2_s^{\ast}}\,{\rm d}x\Big)^{\frac{\theta(p+1)}{2_s^{\ast}}}\\
&\rightarrow0,
\end{align*}
where $\theta\in(0,1)$ such that $\frac{1}{p+1}=\frac{1-\theta}{2}+\frac{\theta}{2_s^{\ast}}$.

{\bf Proof of \eqref{equ3-17}.}
We have
\begin{align*}
\int_{\mathbb{R}^3}&\phi_{u_n}^su_n^2\,{\rm d}x-\int_{\mathbb{R}^3}\phi_{v_n}^sv_n^2\,{\rm d}x-\int_{\mathbb{R}^3}\phi_{w_n}^sw_n^2\,{\rm d}x=2\int_{\mathbb{R}^3}\int_{\mathbb{R}^3}\frac{v_n^2(x)w_n^2(y)}{|x-y|^{3-2s}}\,{\rm d}x\,{\rm d}y\\
&+4\int_{\mathbb{R}^3}\int_{\mathbb{R}^3}\frac{v_n^2(x)v_n(y)w_n(y)}{|x-y|^{3-2s}}\,{\rm d}x\,{\rm d}y+4\int_{\mathbb{R}^3}\int_{\mathbb{R}^3}\frac{v_n(x)w_n(x)w_n^2(y)}{|x-y|^{3-2s}}\,{\rm d}x\,{\rm d}y\\
&+4\int_{\mathbb{R}^3}\int_{\mathbb{R}^3}\frac{v_n(x)v_n(y)w_n(x)w_n(y)}{|x-y|^{3-2s}}\,{\rm d}x\,{\rm d}y.
\end{align*}
Using \eqref{equ3-11}, we have
\begin{align*}
\int_{\mathbb{R}^3}&\int_{\mathbb{R}^3}\frac{v_n^2(x)v_n(y)w_n(y)}{|x-y|^{3-2s}}\,{\rm d}x\,{\rm d}y=\int_{\mathbb{R}^3}\int_{\Omega_n}\frac{v_n^2(x)v_n(y)w_n(y)}{|x-y|^{3-2s}}\,{\rm d}x\,{\rm d}y\leq\int_{\Omega_n}\phi_{u_n}^su_n^2\,{\rm d}y \,{\rm d}x\rightarrow0.
\end{align*}
Similarly, we can show that
\begin{equation*}
\int_{\mathbb{R}^3}\int_{\mathbb{R}^3}\frac{v_n(x)w_n(x)w_n^2(y)}{|x-y|^{3-2s}}\,{\rm d}x\,{\rm d}y=o_n(1)
\end{equation*}
and
\begin{equation*}
\int_{\mathbb{R}^3}\int_{\mathbb{R}^3}\frac{v_n(x)v_n(y)w_n(x)w_n(y)}{|x-y|^{3-2s}}\,{\rm d}x\,{\rm d}y=o_n(1).
\end{equation*}
Therefore, we have proved that
\begin{align*}
\int_{\mathbb{R}^3}\phi_{u_n}u_n^2\,{\rm d}x&-\int_{\mathbb{R}^3}\phi_{v_n}^sv_n^2\,{\rm d}x-\int_{\mathbb{R}^3}\phi_{w_n}^sw_n^2\,{\rm d}x&=2\int_{\mathbb{R}^3}\int_{\mathbb{R}^3}\frac{v_n^2(x)w_n^2(y)}{|x-y|^{3-2s}}\,{\rm d}x\,{\rm d}y+o_n(1)\\
&\geq o_n(1)
\end{align*}
and thus \eqref{equ3-17} is proved.

Hence, by \eqref{equ3-14}, \eqref{equ3-15} and \eqref{equ3-17}, we get
\begin{equation*}
J(u_n)\geq J(v_n)+J(w_n)+o_n(1).
\end{equation*}
Then
\begin{equation*}
c=\lim_{n\rightarrow\infty}J(u_n)\geq\liminf_{n\rightarrow\infty} J(v_n)+\liminf_{n\rightarrow\infty} J(w_n)\geq\alpha+c-\alpha=c,
\end{equation*}
hence
\begin{equation}\label{equ3-18}
\lim_{n\rightarrow\infty}J(v_n)=\alpha,\quad \lim_{n\rightarrow\infty}J(w_n)=c-\alpha.
\end{equation}
Since $u_n\in\mathcal{M}$, $\mathcal{G}(u_n)=0$. By \eqref{equ3-14}-\eqref{equ3-17}, we have
\begin{equation}\label{equ3-19}
0=\mathcal{G}(u_n)\geq\mathcal{G}(v_n)+\mathcal{G}(w_n)+o_n(1).
\end{equation}
We have to discuss the following two cases:\\
Case 1. Up to a subsequence, we may assume that $\mathcal{G}(v_n)\leq0$ or $\mathcal{G}(v_n)\leq0$.

Without loss of generality, we suppose that $\mathcal{G}(v_n)\leq0$, then
\begin{align*}
\frac{6s-3}{2}\int_{\mathbb{R}^3}|(-\Delta)^{\frac{s}{2}}v_n|^2\,{\rm d}x+\frac{4s-3}{2}\int_{\mathbb{R}^3}|v_n|^2\,{\rm d}x+\frac{6s-3}{4}\int_{\mathbb{R}^3}\phi_{v_n}^sv_n^2\,{\rm d}x-\frac{2s(p+1)-3}{p+1}\int_{\mathbb{R}^3}|v_n|^{p+1}\,{\rm d}x\leq0.
\end{align*}
By Lemma \ref{lem2-1}, for any $n$, there exists $\theta_n>0$ such that $(v_n)_{\theta_n}\in\mathcal{M}$ and then $\mathcal{G}((v_n)_{\theta_n})=0$. Hence
\begin{align*}
\frac{6s-3}{2}(\theta_n^{2s(p-1)}-\theta_n^{2s})&\int_{\mathbb{R}^3}|(-\Delta)^{\frac{s}{2}}v_n|^2\,{\rm d}x+\frac{4s-3}{2}(\theta_n^{2s(p-1)}-1)\int_{\mathbb{R}^3}|v_n|^2\,{\rm d}x\\
&+\frac{6s-3}{4}(\theta_n^{2s(p-1)}-\theta_n^{2s})\int_{\mathbb{R}^3}\phi_{v_n}^sv_n^2\,{\rm d}x\leq0
\end{align*}
which implies that $\theta_n\leq1$. Then
\begin{equation}\label{equ3-20}
c\leq I((v_n)_{\theta_n})=J((v_n)_{\theta_n})\leq J(v_n)\rightarrow \alpha<c,
\end{equation}
which is a contradiction.

Case 2. Up to a subsequence, we may assume that $\mathcal{G}(v_n)>0$ and $\mathcal{G}(w_n)>0$.

By \eqref{equ3-19}, we see that $\mathcal{G}(v_n)=o_n(1)$ and $\mathcal{G}(w_n)=o_n(1)$. Repeat the argument of Case 1, we can obtain a contradiction as \eqref{equ3-20}. Thus we suppose that
\begin{equation*}
\lim_{n\rightarrow\infty}\theta_n=\theta_0>1.
\end{equation*}
We have
\begin{align*}
o_n(1)&=\mathcal{G}(v_n)=\frac{6s-3}{2}\int_{\mathbb{R}^3}|(-\Delta)^{\frac{s}{2}}v_n|^2\,{\rm d}x+\frac{4s-3}{2}\int_{\mathbb{R}^3}|v_n|^2\,{\rm d}x\\
&+\frac{6s-3}{4}\int_{\mathbb{R}^3}\phi_{v_n}^sv_n^2\,{\rm d}x-\frac{2s(p+1)-3}{p+1}\int_{\mathbb{R}^3}|v_n|^{p+1}\,{\rm d}x\\
&=\frac{6s-3}{2}(1-\frac{1}{\theta_n^{2s(p-2)}})\int_{\mathbb{R}^3}|(-\Delta)^{\frac{s}{2}}v_n|^2\,{\rm d}x+\frac{4s-3}{2}(1-\frac{1}{\theta_n^{2s(p-1)}})\\
&\int_{\mathbb{R}^3}|v_n|^2\,{\rm d}x+\frac{6s-3}{4}(1-\frac{1}{\theta_n^{2s(p-2)}})\int_{\mathbb{R}^3}\phi_{v_n}^sv_n^2\,{\rm d}x
\end{align*}
which implies that $v_n\rightarrow0$ in $H^s(\mathbb{R}^3)$, this contradicts with \eqref{equ3-18}.

Hence, we conclude that dichotomy cannot occur. As a result, compactness holds for the sequence $\{\rho_n\}$, i.e., there exists $\{y_n\}\subset\mathbb{R}^3$ such that for any $\varepsilon>0$, there exists an $R>0$ satisfying
\begin{equation*}
\liminf_{n\rightarrow\infty}\int_{B_R(y_n)}\rho_n(x)\,{\rm d}x\geq c-\varepsilon.
\end{equation*}
Since $\lim\limits_{n\rightarrow\infty}I(u_n)=\lim\limits_{n\rightarrow\infty}J(u_n)=c$, hence
\begin{align*}
\varepsilon>&c-(c-\varepsilon)\geq \lim_{n\rightarrow\infty}J(u_n)-\liminf_{n\rightarrow\infty}\int_{B_R(y_n)}\rho_n(x)\,{\rm d}x\geq\lim_{n\rightarrow\infty}\Big(J(u_n)-\int_{B_R(y_n)}\rho_n(x)\,{\rm d}x\Big)\\
&=\lim_{n\rightarrow\infty}\int_{\mathbb{R}^3\backslash B_R(y_n)}\rho_n(x)\,{\rm d}x
\end{align*}
which implies the conclusion holds true.
\end{proof}
\begin{remark}\label{rem3-1}
By Lemma \ref{lemp-2-3}, we see that
\begin{equation*}
\int_{\mathbb{R}^3}\phi_{u}^su^2\,{\rm d}x\leq\int_{\mathbb{R}^3}\phi_{u^{\ast}}^s(u^{\ast})^2\,{\rm d}x,
\end{equation*}
we do not know if the converse inequality holds, if it is true, by Lemma \ref{lemp-2-2} and \ref{lemp-2-3}, we can use the symmetric radial decreasing rearrangement technique to simplify the proof of compactness.
\end{remark}
\subsection{Proof of Theorem \ref{thm1-3}.}
Let $\{u_n\}\subset\mathcal{M}$ be a minimizing sequence for $c$, by Lemma \ref{lem3-4}, there exists $\{y_n\}\subset\mathbb{R}^3$ such that for any $\varepsilon>0$, there exists an $R>0$ satisfying
\begin{equation}\label{equ3-21}
\int_{\mathbb{R}^3\backslash B_R(y_n)}\Big(\int_{\mathbb{R}^3}\frac{|u_n(x)-u_n(y)|^2}{|x-y|^{3+2s}}\,{\rm d}y+u_n^2\Big)\,{\rm d}x\leq \varepsilon.
\end{equation}
Define $\widetilde{u_n}(x)=u_n(x-y_n)\in H^s(\mathbb{R}^3)$, then $\phi_{\widetilde{u_n}}^s=\phi_{u_n}^s(\cdot-y_n)$ and thus $\widetilde{u_n}\in\mathcal{M}$ and also $J(u_n)=J(\widetilde{u_n})$. This means that $\widetilde{u_n}$ is also a minimizing sequence for $c$. Hence, by \eqref{equ3-21}, we have for any $\varepsilon>0$, there exists an $R>0$ such that
\begin{equation}\label{equ3-22}
\int_{\mathbb{R}^3\backslash B_R(0)}\Big(\int_{\mathbb{R}^3}\frac{|\widetilde{u_n}(x)-\widetilde{u_n}(y)|^2}{|x-y|^{3+2s}}\,{\rm d}y+\widetilde{u_n}^2\Big)\,{\rm d}x\leq \varepsilon.
\end{equation}
Since $\widetilde{u_n}$ is bounded in $H^s(\mathbb{R}^3)$, up to a subsequence, we may assume that there exists $\widetilde{u}\in H^s(\mathbb{R}^3)$ such that
\begin{equation}\label{equ3-23}
\left\{
  \begin{array}{ll}
   \widetilde{u_n}\rightharpoonup \widetilde{u}, & \hbox{in $H^s(\mathbb{R}^3)$, } \\
   \widetilde{u_n}\rightarrow \widetilde{u}, & \hbox{in $L_{{\rm loc}}^r(\mathbb{R}^3)$ with $1\leq r<2_s^{\ast}$,} \\
    \widetilde{u_n}\rightarrow \widetilde{u}, & \hbox{a.e. in $\mathbb{R}^3$.}
  \end{array}
\right.
\end{equation}
By Fatou's Lemma and \eqref{equ3-22}, we get
\begin{equation}\label{equ3-24}
\int_{\mathbb{R}^3\backslash B_R(0)}\Big(\int_{\mathbb{R}^3}\frac{|\widetilde{u}(x)-\widetilde{u}(y)|^2}{|x-y|^{3+2s}}\,{\rm d}y+\widetilde{u}^2\Big)\,{\rm d}x\leq \varepsilon.
\end{equation}
By \eqref{equ3-22}-\eqref{equ3-24}, and the Sobolev embedding theorem, we have that for any $r\in[2,2_s^{\ast})$ and any $\varepsilon>0$, there exists a $C>0$ such that
\begin{align*}
\int_{\mathbb{R}^3}|\widetilde{u_n}-\widetilde{u}|^r\,{\rm d}x&=\int_{B_R(0)}|\widetilde{u_n}-\widetilde{u}|^r\,{\rm d}x+\int_{\mathbb{R}^3\backslash B_R(0)}|\widetilde{u_n}-\widetilde{u}|^r\,{\rm d}x\\
&\leq\varepsilon+C(\|\widetilde{u_n}\|_{H^s(\mathbb{R}^3\backslash B_R(0))}+\|\widetilde{u}\|_{H^s(\mathbb{R}^3\backslash B_R(0))})\\
&\leq\varepsilon+C\Big(\int_{\mathbb{R}^3\backslash B_R(0)}\Big(\int_{\mathbb{R}^3}\frac{|\widetilde{u_n}(x)-\widetilde{u_n}(y)|^2}{|x-y|^{3+2s}}\,{\rm d}y\\
&+\int_{\mathbb{R}^3\backslash B_R(0)}\int_{\mathbb{R}^3}\frac{|\widetilde{u}(x)-\widetilde{u}(y)|^2}{|x-y|^{3+2s}}\,{\rm d}y\Big)\\
&\leq(1+2C)\varepsilon,
\end{align*}
where $C>0$ is the constant of the embedding $H^s(B_R(0))\hookrightarrow L^r(B_R(0))$. Hence, we have proved that
\begin{equation}\label{equ3-25}
\widetilde{u_n}\rightarrow\widetilde{u}\quad \text{in}\,\, L^r(\mathbb{R}^3)\,\, \text{for any}\,\, r\in[2,2_s^{\ast}).
\end{equation}
Since $\widetilde{u_n}\in\mathcal{M}$, by Lemma \ref{lem3-2}, $\|\widetilde{u_n}\|_{p+1}\geq C>0$, hence $\|\widetilde{u}\|_{p+1}\geq C>0$, and as a result, $\widetilde{u}\neq0$.

Finally, we show that $\widetilde{u_n}\rightarrow\widetilde{u}$ in $H^s(\mathbb{R}^3)$. By \eqref{equ2-6} and \eqref{equ3-25}, we deduce that
\begin{equation*}
\phi_{\widetilde{u_n}}^s\rightarrow\phi_{\widetilde{u}}^s\quad \text{in}\,\, \mathcal{D}^{s,2}(\mathbb{R}^3),
\end{equation*}
and thus
\begin{equation}\label{equ3-26}
\int_{\mathbb{R}^3}\phi_{\widetilde{u_n}}^s\widetilde{u_n}^2\,{\rm d}x\rightarrow\int_{\mathbb{R}^3}\phi_{\widetilde{u}}^s\widetilde{u}^2\,{\rm d}x.
\end{equation}
From the weak semi-continuousness of norm, it is easy to verify that
\begin{equation*}
\mathcal{G}(\widetilde{u})\leq\liminf_{n\rightarrow\infty}\mathcal{G}(\widetilde{u_n})=0.
\end{equation*}
Let us set
\begin{equation*}
a=\int_{\mathbb{R}^3}\int_{\mathbb{R}^3}\frac{|\widetilde{u}(x)-\widetilde{u}(y)|^2}{|x-y|^{3+2s}}\,{\rm d}x\,{\rm d}y,\quad \widetilde{a}=\liminf_{n\rightarrow\infty}\int_{\mathbb{R}^3}\int_{\mathbb{R}^3}\frac{|\widetilde{u_n}(x)-\widetilde{u_n}(y)|^2}{|x-y|^{3+2s}}\,{\rm d}x\,{\rm d}y.
\end{equation*}
Obviously, $a\leq \widetilde{a}$. We suppose that $a<\widetilde{a}$, then $I(\widetilde{u})<\liminf\limits_{n\rightarrow\infty}I(\widetilde{u_n})=c$ and $\mathcal{G}(\widetilde{u})<0$. By Lemma \ref{lem3-2}, there exists a $0<\theta<1$ such that $\widetilde{u}_{\theta}\in\mathcal{M}$. Therefore, we have
\begin{equation*}
c\leq I(\widetilde{u}_{\theta})=J(\widetilde{u}_{\theta})<J(\widetilde{u})\leq\liminf_{n\rightarrow\infty}J(\widetilde{u_n})=\liminf_{n\rightarrow\infty}I(\widetilde{u_n})=c,
\end{equation*}
which means a contradiction. Hence $a=\widetilde{a}$. So $\widetilde{u_n}\rightarrow\widetilde{u}$ in $H^s(\mathbb{R}^3)$ and we can conclude that $\widetilde{u}\in\mathcal{M}$ and $I(\widetilde{u})=c$.

Finally, we only need to prove that the ground state solution is nonnegative. Put $u_{+}=\max\{u, 0\}$,
the positive part of $u$. We note that all the calculations above can be repeated word by word, replacing $I$ with the functional
\begin{equation*}
I_{+}(u)=\frac{1}{2}\int_{\mathbb{R}^3}|(-\Delta )^{\frac{s}{2}}u|^2+V(x)u^2\,{\rm d}x+\frac{1}{4}\int_{\mathbb{R}^3}\phi_u^su^2\,{\rm d}x-\frac{1}{p+1}\int_{\mathbb{R}^3}u_{+}^{p+1}\,{\rm d}x.
\end{equation*}
In this way we get a ground state solution $u$ of the equation
\begin{equation*}
(-\Delta)^su+u+\phi_u^su=(u_{+})^{p}.
\end{equation*}
Multiplying the above equation by $u_{-}$ and we integrate over $\mathbb{R}^3$, we obtain
\begin{equation*}
\int_{\mathbb{R}^3}(-\Delta)^su u_{-}\,{\rm d}x=-\int_{\mathbb{R}^3}(1+\phi_u^s)(u_{-})^2\,{\rm d}x\leq0.
\end{equation*}
Hence, by integrate by parts we get
\begin{equation*}
\int_{\mathbb{R}^3}\int_{\mathbb{R}^3}\frac{(u(x)-u(y))(u_{-}(x)-u_{-}(y))}{|x-y|^{3+2s}}\,{\rm d}x\,{\rm d}y\leq0.
\end{equation*}
Similarly as the proof of Lemma \ref{lemr-2-1}, we deduce that
\begin{equation*}
\int_{\mathbb{R}^3}\int_{\mathbb{R}^3}\frac{|u_{-}(x)-u_{-}(y)|^2}{|x-y|^{3+2s}}\,{\rm d}x\,{\rm d}y\leq0
\end{equation*}
which implies that $u_{-}=0$. Thus $u\geq0$.

\section{Nonconstant potential case}

Let $\Lambda=[\delta,1]$, $\delta\in(0,1)$ is a positive constant. We consider a family of functionals $I_{V,\lambda}: H^s(\mathbb{R}^3)\rightarrow\mathbb{R}$ defined by
\begin{equation*}
I_{V,\lambda}(u)=\frac{1}{2}\int_{\mathbb{R}^3}|(-\Delta )^{\frac{s}{2}}u|^2+V(x)u^2\,{\rm d}x+\frac{1}{4}\int_{\mathbb{R}^3}\phi_u^su^2\,{\rm d}x-\frac{\lambda}{p+1}\int_{\mathbb{R}^3}|u|^{p+1}\,{\rm d}x.
\end{equation*}
Let $I_{V,\lambda}(u)=A(u)-\lambda B(u)$, where
\begin{equation*}
A(u)=\frac{1}{2}\int_{\mathbb{R}^3}|(-\Delta )^{\frac{s}{2}}u|^2+V(x)u^2\,{\rm d}x+\frac{1}{4}\int_{\mathbb{R}^3}\phi_u^su^2\,{\rm d}x\rightarrow+\infty,\quad \text{as}\,\, \|u\|\rightarrow\infty
\end{equation*}
and
\begin{equation*}
B(u)=\frac{1}{p+1}\int_{\mathbb{R}^3}|u|^{p+1}\,{\rm d}x.
\end{equation*}

We first verify all the conditions of Theorem \ref{thm1-2}.
\begin{lemma}\label{lem4-1}
Suppose that $(V_1)$ and $(V_2)$ hold and $2<p<2_s^{\ast}-1$. Then\\
$(i)$ there exists a $v_0\in H^s(\mathbb{R}^3)\backslash\{0\}$ such that $I_{V,\lambda}(v_0)<0$ for any $\lambda\in\Lambda$;\\
$(ii)$ $c_{\lambda}=\inf\limits_{\gamma\in\Gamma}\max\limits_{t\in[0,1]}I_{V,\lambda}(\gamma(t))>\max\{I_{V,\lambda}(0),I_{V,\lambda}(v_0)\}$ for all $\lambda\in\Lambda$, where
$\Gamma=\{\gamma\in C([0,1],H^s(\mathbb{R}^3))\,\,\Big|\,\, \gamma(0)=0, \gamma(1)=v_0\}$.
\end{lemma}
\begin{proof}
For $v\in H^s(\mathbb{R}^3)\backslash\{0\}$ fixed, and any $\lambda\in\Lambda$, we have
\begin{equation*}
I_{V,\lambda}(v)\leq I_{\lambda}^{\infty}(v)=\frac{1}{2}\int_{\mathbb{R}^3}|(-\Delta )^{\frac{s}{2}}v|^2+V_{\infty}v^2\,{\rm d}x+\frac{1}{4}\int_{\mathbb{R}^3}\phi_v^sv^2\,{\rm d}x-\frac{\delta}{p+1}\int_{\mathbb{R}^3}|v|^{p+1}\,{\rm d}x.
\end{equation*}
Set $v_{\theta}=\theta^{2s}v(\theta x)$, $\forall \theta>0$. By Lemma \ref{lem3-1}, then
\begin{equation*}
I_{\lambda}^{\infty}(v_{\theta})\rightarrow-\infty\quad \text{as}\,\, \theta\rightarrow+\infty.
\end{equation*}
Hence, take $v_0\:= v_{\theta}$ for $\theta$ large, we have that $I_{V,\lambda}(v_0)\leq I_{\lambda}^{\infty}(v_0)<0$.

$(ii)$ By the Sobolev embedding theorem, we have
\begin{equation*}
I_{V,\lambda}(v)\geq\frac{1}{2}\|v\|^2-\frac{1}{p+1}\int_{\mathbb{R}^3}|v|^{p+1}\,{\rm d}x\geq\frac{1}{2}\|v\|^2-\frac{C}{p+1}\|v\|^{p+1}.
\end{equation*}
Since $p>2$, we see that there exist $\beta>0$ and $r_0>0$ such that
\begin{equation*}
I_{V,\lambda}(v)\geq\beta>0,\quad \forall \,\, \|v\|=r_0, \,\, \text{for any}\,\, \lambda\in\Lambda.
\end{equation*}
Therefore, for any $\gamma\in\Gamma$, clearly, there must exists a $t_0\in(0,1)$ such that $\|\gamma(t_0)\|=r_0$ and thus
\begin{equation*}
\max_{t\in[0,1]}I_{V,\lambda}(\gamma(t))\geq I_{V,\lambda}(\gamma(t_0))\geq\beta>\max\{I_{V,\lambda}(0),\,\, I_{V,\lambda}(v_0)\}
\end{equation*}
which implies that $c_{\lambda}>0$.
\end{proof}

Using Theorem \ref{thm1-2} , we get that for a.e. $\lambda\in\Lambda$, there exists a bounded sequence $\{u_n\}\subset H^s(\mathbb{R}^3)$ (for simplicity, we denote $\{u_n\}$ instead of $\{u_n(\lambda)\}$ ) such that
\begin{equation*}
I_{V,\lambda}(u_n)\rightarrow c_{\lambda},\quad I_{V,\lambda}'(u_n)\rightarrow0.
\end{equation*}
Next, we aim to prove the strongly convergence of the above sequence $\{u_n\}$ in $H^s(\mathbb{R}^3)$.

By Theorem \ref{thm1-3}, we see that for any $\lambda\in\Lambda$, the limit problem corresponded to problem \eqref{main}
\begin{equation}\label{equ4-1}
\left\{
  \begin{array}{ll}
    (-\Delta)^su+V_{\infty}u+\phi(x)u=\lambda|u|^{p-1}u, & \hbox{in $\mathbb{R}^3$,} \\
    (-\Delta)^s\phi=u^2,& \hbox{in $\mathbb{R}^3$}
  \end{array}
\right.
\end{equation}
with $2<p<2_s^{\ast}-1$, has a ground state solution in $H^s(\mathbb{R}^3)$, i.e. for any $\lambda\in\Lambda$,
\begin{equation*}
c_{\lambda}^{\infty}\:=\inf_{u\in\mathcal{M}_{\lambda}^{\infty}}I_{\lambda}^{\infty}(u)
\end{equation*}
can be achieved at some $u_{\lambda}^{\infty}\in\mathcal{M}_{\lambda}^{\infty}$ and $(I_{\lambda}^{\infty})'(u_{\lambda}^{\infty})=0$, where
\begin{equation*}
I_{\lambda}^{\infty}(u)=\frac{1}{2}\int_{\mathbb{R}^3}|(-\Delta )^{\frac{s}{2}}u|^2+V_{\infty}u^2\,{\rm d}x+\frac{1}{4}\int_{\mathbb{R}^3}\phi_u^su^2\,{\rm d}x-\frac{\lambda}{p+1}\int_{\mathbb{R}^3}|u|^{p+1}\,{\rm d}x,
\end{equation*}
\begin{equation*}
\mathcal{M}_{\lambda}^{\infty}=\{u\in H^s(\mathbb{R}^3)\backslash\{0\}\,\,\Big|\,\, \mathcal{G}_{\lambda}^{\infty}(u)=0\}
\end{equation*}
and
\begin{align*}
\mathcal{G}_{\lambda}^{\infty}(u)&=\frac{6s-3}{2}\int_{\mathbb{R}^3}|(-\Delta)^{\frac{s}{2}}u|^2\,{\rm d}x+\frac{4s-3}{2}\int_{\mathbb{R}^3}V_{\infty}|u|^2\,{\rm d}x\\
&+\frac{6s-3}{4}\int_{\mathbb{R}^3}\phi_{u}^su^2\,{\rm d}x-\lambda\frac{2s(p+1)-3}{p+1}\int_{\mathbb{R}^3}|u|^{p+1}\,{\rm d}x.
\end{align*}

\begin{lemma}\label{lem4-2}(Global Compactness)
Assume that $(V_1)$ and $(V_2)$ hold and for any $\lambda\in\Lambda$, let $\{u_n\}$ be a bounded $(PS)_{c_{\lambda}}$ sequence for the functional $I_{V,\lambda}$. Then there exist subsequence of $\{u_n\}$, still denoted $\{u_n\}$ and integer $l\in\mathbb{N}\cup\{0\}$, sequence $\{y_n^k\}\subset\mathbb{R}^3$, $w^k\in H^s(\mathbb{R}^3)$ for $1\leq k\leq l$ such that\\
$(i)$ $u_n\rightharpoonup u_0$ with $I_{V,\lambda}'(u_0)=0$;\\
$(ii)$ $|y_n^k|\rightarrow+\infty$ and $|y_n^k-y_n^{k'}|\rightarrow+\infty$ for $k\neq k'$;\\
$(iii)$ $w^k\neq0$ and $(I_{\lambda}^{\infty})'(w^k)=0$ for $1\leq k\leq l$;\\
$(iv)$ $\|u_n-u_0-\sum_{k=1}^lw^k(\cdot-y_n^k)\|\rightarrow0$;\\
$(v)$ $I_{V,\lambda}(u_n)\rightarrow I_{V,\lambda}(u_0)+\sum_{k=1}^l I_{\lambda}^{\infty}(w^k)$.\\
Here we agree that in the case $l=0$ the above holds without $\{w^k\}$ and $\{y_n^k\}$.
\end{lemma}

\begin{proof}
Step 1. Since $\{u_n\}$ is bounded in $H^s(\mathbb{R}^3)$, we may assume that, up to a subsequence, $u_n\rightharpoonup u_0$ weakly in $H^s(\mathbb{R}^3)$, $u_n\rightarrow u_0$ in $L_{loc}^r(\mathbb{R}^3)$ for $1\leq r<2_s^{\ast}$ and $u_n\rightarrow u_0$ a.e. in $\mathbb{R}^3$. Let us prove that $I_{V,\lambda}'(u_0)=0$. Noting that $C_0^{\infty}(\mathbb{R}^3)$ is dense in $H^s(\mathbb{R}^3)$, it sufficient to check that $\langle I_{V,\lambda}'(u_0),\varphi\rangle=0$ for all $\varphi\in C_0^{\infty}(\mathbb{R}^3)$. Observe that
\begin{align}\label{equ4-2}
\langle I_{V,\lambda}'(u_n),\varphi\rangle-\langle I_{V,\lambda}'(u_0),\varphi\rangle&=\langle u_n-u_0,\varphi\rangle+\int_{\mathbb{R}^3}(\phi_{u_n}^su_n-\phi_{u_0}^su_0)\varphi\,{\rm d}x\nonumber\\
&-\lambda\int_{\mathbb{R}^3}(|u_n|^{p-1}u_n-|u|^{p-1}u)\varphi\,{\rm d}x.
\end{align}
In view of $u_n\rightharpoonup u_0$ weakly in $H^s(\mathbb{R}^3)$, then $\langle u_n-u_0,\varphi\rangle\rightarrow0$.

By H\"{o}lder's inequality and $|a^m-b^m|\leq L|a-b|^m$ for $a,b\geq0$, $m\geq1$ and $L\geq1$, we have
\begin{align*}
&\Big|\int_{\mathbb{R}^3}(|u_n|^{p-1}u_n-|u_0|^{p-1}u_0)\varphi\,{\rm d}x\Big|\leq\int_{\mathbb{R}^3}|u_n|^{p-1}|u_n-u_0||\varphi|\,{\rm d}x+\int_{\mathbb{R}^3}\Big||u_n|^{p-1}-|u_0|^{p-1}\Big||u_0\varphi|\,{\rm d}x\\
&\leq\|\varphi\|_{\infty}\|u_n\|_{L^p}^{p-1}\Big(\int_{{\rm supp}\varphi}|u_n-u_0|^p\,{\rm d}x\Big)^{\frac{1}{p}}+C\|\varphi\|_{\infty}\|u_0\|_{L^p}\Big(\int_{{\rm supp}\varphi}|u_n-u_0|^p\,{\rm d}x\Big)^{\frac{p-1}{p}}\\
&\rightarrow0.
\end{align*}
By \eqref{equ2-1}, $(ii)$ of Lemma \ref{lem2-2} and H\"{o}lder's inequality, we deduce that
\begin{align*}
&\int_{\mathbb{R}^3}\phi_{u_n}^su_n\varphi\,{\rm d}x-\int_{\mathbb{R}^3}\phi_{u_0}^su_0\varphi\,{\rm d}x=\int_{\mathbb{R}^3}\phi_{u_n-u_0}^s(u_n-u_0)\varphi\,{\rm d}x+o_n(1)\\
&\leq\|\varphi\|_{\infty}\Big(\int_{{\rm supp}\varphi}|u_n-u_0|^{\frac{6}{3+2s}}\,{\rm d}x\Big)^{\frac{3+2s}{6}}\Big(\int_{\mathbb{R}^3}|\phi_{u_n-u_0}^s|^{\frac{6}{3-2s}}\,{\rm d}x\Big)^{\frac{3-2s}{6}}+o_n(1)\\
&\leq S_s\|\varphi\|_{\infty}\Big(\int_{{\rm supp}\varphi}|u_n-u_0|^{\frac{6}{3+2s}}\,{\rm d}x\Big)^{\frac{3+2s}{6}}\|\phi_{u_n-u_0}^s\|_{\mathcal{D}^{s,2}}+o_n(1)\\
&\leq S_s\|\varphi\|_{\infty}\Big(\int_{{\rm supp}\varphi}|u_n-u_0|^{\frac{6}{3+2s}}\,{\rm d}x\Big)^{\frac{3+2s}{6}}\|u_n-u_0\|^2+o_n(1)\\
&\leq CS_s\Big(\int_{{\rm supp}\varphi}|u_n-u_0|^{\frac{6}{3+2s}}\,{\rm d}x\Big)^{\frac{3+2s}{6}}\\
&\rightarrow0.
\end{align*}
Thus recalling that $I_{V,\lambda}'(u_n)\rightarrow0$, we indeed have $I_{V,\lambda}'(u_0)=0$. Next we prove that $I_{V,\lambda}(u_0)\geq0$. Set
\begin{equation*}
a_0=\int_{\mathbb{R}^3}|(-\Delta)^{\frac{s}{2}}u_0|^2\,{\rm d}x,\,\, b_0=\int_{\mathbb{R}^3}V(x)|u_0|^2\,{\rm d}x,\,\,c_0=\int_{\mathbb{R}^3}\phi_{u_0}^su_0^2\,{\rm d}x
\end{equation*}
\begin{equation*}
d_0=\frac{\lambda}{p+1}\int_{\mathbb{R}^3}|u_0|^{p+1}\,{\rm d}x,\,\, e_0=\int_{\mathbb{R}^3}(x,\nabla V(x))|u_0|^2\,{\rm d}x.
\end{equation*}
Then, we have the following linear systems of $a_0,b_0,c_0,d_0,e_0$:
\begin{equation}\label{equ4-3}
\left\{
  \begin{array}{ll}
    \frac{1}{2}a_0+\frac{1}{2}b_0+\frac{1}{4}c_0-d_0=I_{V,\lambda}(u_0), & \hbox{} \\
    a_0+b_0+c_0-(p+1)d_0=0, & \hbox{} \\
    \frac{3-2s}{2}a_0+\frac{3}{2}b_0+\frac{3+2s}{4}c_0-3d_0+\frac{1}{2}e_0=0, & \hbox{}
  \end{array}
\right.
\end{equation}
where the first equation comes from the definition of $I_{V,\lambda}(u_0)$, the second one is $\langle I_{V,\lambda}'(u_0),u_0\rangle=0$, and the last one is from the Pohozaev identity in Proposition \ref{pro2-1}. From the above relation, (proceed as: the second equation minus the two times of the first equation, the last equation minus the $(3-2s)$ times of the first equation), using the assumption $(V_1)$, we can deduce that
\begin{equation*}
(6s-3)I_{V,\lambda}(u_0)=\frac{1}{2}(2sb_0+e_0)+2s(p-2)d_0\geq0.
\end{equation*}

Step 2. Set $v_n^1=u_n-u_0$, then we have $v_n^1\rightharpoonup 0$ weakly in $H^s(\mathbb{R}^3)$. Let us define
\begin{equation*}
\delta:=\limsup_{n\rightarrow+\infty}\sup_{y\in\mathbb{R}^3}\int_{B_1(y)}|v_n^1|^2\,{\rm d}x.
\end{equation*}

Case 1 (Vanishing). $\delta=0$. Namely,
\begin{equation}\label{equ4-4}
\sup_{y\in\mathbb{R}^3}\int_{B_1(y)}|v_n^1|^2\,{\rm d}x\rightarrow0.
\end{equation}
Then $v_n^1\rightarrow0$ in $H^s(\mathbb{R}^3)$ and Lemma \ref{lem4-2} holds with $l=0$.

Indeed, by similar computation as \eqref{equ4-2}, we have
\begin{align*}
\|v_n^1\|^2&=\langle I_{V,\lambda}'(u_n),v_n^1\rangle-\langle I_{V,\lambda}'(u_0),v_n^1\rangle+\int_{\mathbb{R}^3}\phi_{u_0}^su_0v_n^1\,{\rm d}x\\
&-\int_{\mathbb{R}^3}\phi_{u_n}^su_nv_n^1\,{\rm d}x+\lambda\int_{\mathbb{R}^3}(|u_n|^{p-1}u_n-|u_0|^{p-1}u_0)v_n^1\,{\rm d}x\\
&=\int_{\mathbb{R}^3}\phi_{u_0}^su_0v_n^1\,{\rm d}x-\int_{\mathbb{R}^3}\phi_{u_n}^su_nv_n^1\,{\rm d}x+\lambda\int_{\mathbb{R}^3}|u_n|^{p-1}u_nv_n^1\,{\rm d}x\\
&-\lambda\int_{\mathbb{R}^3}|u_0|^{p-1}u_0v_n^1\,{\rm d}x+o_n(1).
\end{align*}
By \eqref{equ4-4} and using the vanishing Lemma \ref{lem2-3}, we get $v_n^1\rightarrow0$ in $L^r(\mathbb{R}^3)$ for $2<r<2_s^{\ast}$.
By \eqref{equ2-1}, $(ii)$ of Lemma \ref{lem2-2} and H\"{o}lder's inequality, we have that
\begin{align*}
&\int_{\mathbb{R}^3}\phi_{u_n}^su_nv_n^1\,{\rm d}x-\int_{\mathbb{R}^3}\phi_{u_0}^su_0v_n^1\,{\rm d}x=\int_{\mathbb{R}^3}\phi_{v_n^1}^s(v_n^1)^2\,{\rm d}x+o_n(1)\\
&\leq\Big(\int_{\mathbb{R}^3}|v_n^1|^{\frac{12}{3+2s}}\,{\rm d}x\Big)^{\frac{3+2s}{6}}\Big(\int_{\mathbb{R}^3}|\phi_{v_n^1}^s|^{\frac{6}{3-2s}}\,{\rm d}x\Big)^{\frac{3-2s}{6}}+o_n(1)\\
&\leq S_s\|v_n^1\|_{L^{\frac{12}{3+2s}}}^2\|\phi_{v_n^1}^s\|_{\mathcal{D}^{s,2}}+o_n(1)\\
&\leq CS_s\|v_n^1\|_{L^{\frac{12}{3+2s}}}^2+o_n(1)\\
&\rightarrow0.
\end{align*}
Using H\"{o}lder's inequality, we deduce that
\begin{align*}
\Big|\int_{\mathbb{R}^3}(|u_n|^{p-1}u_n-|u_0|^{p-1}u_0)v_n^1\,{\rm d}x\Big|\leq(\|u_n\|_{p+1}^p+\|u_0\|_{p+1}^p)\|v_n^1\|_{p+1}\rightarrow0.
\end{align*}
Therefore, $\|v_n^1\|\rightarrow0$.

Case 2 (Nonvanishing). $\delta>0$. we may assume that there exists $y_n^1\in\mathbb{R}^3$ such that
\begin{equation*}
\int_{B_1(y_n^1)}|v_n^1|^2\,{\rm d}x>\frac{\delta}{2}>0,
\end{equation*}
Then, after extracting a subsequence if necessary, we have for a $w^1\in H^s(\mathbb{R}^3)$,
\begin{equation*}
|y_n^1|\rightarrow+\infty,\,\,v_n^1(\cdot+y_n^1)\rightharpoonup w^1\neq0,\,\,(I_{\lambda}^{\infty})'(w^1)=0.
\end{equation*}

Let us define $\widetilde{v_n^1}(\cdot):=v_n^1(\cdot+y_n^1)$. Then $\widetilde{v_n^1}$ is bounded in $H^s(\mathbb{R}^3)$ and we may assume that $\widetilde{v_n^1}\rightharpoonup w^1$ in $H^s(\mathbb{R}^3)$ and $\widetilde{v_n^1}\rightarrow w^1$ in $L_{loc}^r(\mathbb{R}^3)$ and $\widetilde{v_n^1}\rightarrow w^1$ a.e. in $\mathbb{R}^3$. Since
\begin{equation*}
\int_{B_1(0)}|\widetilde{v_n^1}|^2\,{\rm d}x>\frac{\delta}{2},
\end{equation*}
then
\begin{equation*}
\int_{B_1(0)}|w^1|^2\,{\rm d}x>\frac{\delta}{2},
\end{equation*}
and $w^1\neq0$. But it follows from $v_n^1\rightharpoonup0$ in $H^s(\mathbb{R}^3)$ that $\{y_n^1\}$ must be unbounded. Up to a subsequence, we suppose that $|y_n^1|\rightarrow+\infty$. Now we will prove $(I_{\lambda}^{\infty})'(w^1)=0$. Similar to the proof of \eqref{equ4-2}, we see that
\begin{equation*}
\langle(I_{\lambda}^{\infty})'(\widetilde{v_n^1}),\varphi\rangle-\langle(I_{\lambda}^{\infty})'(w^1),\varphi\rangle\rightarrow0
\end{equation*}
for any fixed $\varphi\in C_0^{\infty}(\mathbb{R}^3)$. Therefore, it suffices to show that $\langle(I_{\lambda}^{\infty})'(\widetilde{v_n^1}),\varphi\rangle\rightarrow0$ for any fixed $\varphi\in C_0^{\infty}(\mathbb{R}^3)$. We have
\begin{align}\label{equ4-5}
&\langle I_{V,\lambda}'(v_n^1),\varphi(\cdot-y_n^1)\rangle=\int_{\mathbb{R}^3}\int_{\mathbb{R}^3}\frac{(v_n^1(x)-v_n^1(y))(\varphi(x-y_n^1)-\varphi(y-y_n^1))}{|x-y|^{3+2s}}\,{\rm d}x{\rm d}y\nonumber\\
&+\int_{\mathbb{R}^3}V(x)v_n^1(x)\varphi(x-y_n^1)\,{\rm d}x+\int_{\mathbb{R}^3}\phi_{v_n^1}^sv_n^1(x)\varphi(x-y_n^1)\,{\rm d}x\nonumber\\
&-\lambda\int_{\mathbb{R}^3}|v_n^1(x)|^{p-1}v_n^1(x)\varphi(x-y_n^1)\,{\rm d}x\nonumber\\
&=\int_{\mathbb{R}^3}\int_{\mathbb{R}^3}\frac{(\widetilde{v_n^1}(x)-\widetilde{v_n^1}(y))(\varphi(x)-\varphi(y))}{|x-y|^{3+2s}}\,{\rm d}x{\rm d}y\nonumber\\
&+\int_{\mathbb{R}^3}V(x+y_n^1)\widetilde{v_n^1}(x)\varphi(x)\,{\rm d}x+\int_{\mathbb{R}^3}\phi_{\widetilde{v_n^1}}^s\widetilde{v_n^1}\varphi(x)\,{\rm d}x-\lambda\int_{\mathbb{R}^3}|\widetilde{v_n^1}(x)|^{p-1}\widetilde{v_n^1}(x)\varphi(x)\,{\rm d}x.
\end{align}
Since $v_n^1\rightharpoonup0$ in $H^s(\mathbb{R}^3)$, similar as the proof of \eqref{equ4-2}, we obtain that
\begin{equation*}
\langle I_{V,\lambda}'(v_n^1),\varphi(\cdot-y_n^1)\rangle-\langle I_{V,\lambda}'(0),\varphi(\cdot-y_n^1)\rangle\rightarrow0.
\end{equation*}
which implies that
\begin{equation*}
\langle I_{V,\lambda}'(v_n^1),\varphi(\cdot-y_n^1)\rangle\rightarrow0,
\end{equation*}
from \eqref{equ4-5}, we get
\begin{align}\label{equ4-6}
\int_{\mathbb{R}^3}\int_{\mathbb{R}^3}&\frac{(\widetilde{v_n^1}(x)-\widetilde{v_n^1}(y))(\varphi(x)-\varphi(y))}{|x-y|^{3+2s}}\,{\rm d}x{\rm d}y+\int_{\mathbb{R}^3}V(x+y_n^1)\widetilde{v_n^1}(x)\varphi(x)\,{\rm d}x\nonumber\\
&+\int_{\mathbb{R}^3}\phi_{\widetilde{v_n^1}}^s\widetilde{v_n^1}\varphi(x)\,{\rm d}x-\lambda\int_{\mathbb{R}^3}|\widetilde{v_n^1}(x)|^{p-1}\widetilde{v_n^1}(x)\varphi(x)\,{\rm d}x
\rightarrow0.
\end{align}

By $(V_2)$, for every $\varepsilon>0$, there exists $R>0$ such that
\begin{equation*}
|V(x)-V_{\infty}|< \varepsilon,\quad \forall\, |x|\geq R.
\end{equation*}
We may assume that ${\rm supp}\varphi\subset B_{R_0}(0)$ with some $R_0>0$. Thus, choosing $n$ large enough such that $|y_n^1|\geq R+R_0$, we have
\begin{equation*}
|x+y_n^1|\geq |y_n^1|-|x|\geq |y_n^1|-R_0\geq R
\end{equation*}
which implies that
\begin{equation*}
|V(x+y_n^1)-V_{\infty}|<\varepsilon
\end{equation*}
for $n$ large enough. Therefore, for $n$ large enough, by H\"{o}lder's inequality, we have
\begin{align}\label{equ4-7}
\Big|\int_{\mathbb{R}^3}V(x+y_n^1)\widetilde{v_n^1}(x)\varphi(x)\,{\rm d}x&-\int_{\mathbb{R}^3}V_{\infty}\widetilde{v_n^1}(x)\varphi(x)\,{\rm d}x\Big|\nonumber\\
&\leq\int_{{\rm supp}\varphi}|V(x+y_n^1)-V_{\infty}||\widetilde{v_n^1}(x)\varphi(x)|\,{\rm d}x\nonumber\\
&\leq C\varepsilon.
\end{align}
Thus, we use \eqref{equ4-6} minus $\langle(I_{\lambda}^{\infty})'(\widetilde{v_n^1}),\varphi\rangle$, by \eqref{equ4-7}, we deduce that
\begin{equation*}
\langle(I_{\lambda}^{\infty})'(\widetilde{v_n^1}),\varphi\rangle\rightarrow0.
\end{equation*}

Step 3. In the following, we show that
\begin{equation}\label{equ4-8}
I_{V,\lambda}(v_n^1)\rightarrow c_{\lambda}-I_{V,\lambda}(u_0),\quad I_{V,\lambda}(u_n)-I_{V,\lambda}(u_0)-I_{\lambda}^{\infty}(v_n^1)\rightarrow0.
\end{equation}

Indeed, from Brezis-Lieb Lemma and $(i)$ of Lemma \ref{lem2-2}, we get
\begin{equation}\label{equ4-9}
\|v_n^1\|^2=\|u_n\|^2-\|u_0\|^2+o_n(1),\quad \|v_n^1\|_{L^{p+1}}^{p+1}=\|u_n\|_{L^{p+1}}^{p+1}-\|u_0\|_{L^{p+1}}^{p+1}+o_n(1),
\end{equation}
\begin{equation}\label{equ4-10}
\int_{\mathbb{R}^3}\phi_{v_n^1}^s(v_n^1)^2\,{\rm d}x=\int_{\mathbb{R}^3}\phi_{u_n}^su_n^2\,{\rm d}x-\int_{\mathbb{R}^3}\phi_{u_0}^su_0^2\,{\rm d}x+o_n(1)
\end{equation}
which implies that
\begin{equation*}
I_{V,\lambda}(v_n^1)\rightarrow c_{\lambda}+I_{V,\lambda}(u_0).
\end{equation*}
By simple computation, we  deduce that
\begin{align*}
I_{V,\lambda}&(u_n)-I_{\lambda}^{\infty}(v_n^1)=\langle u_n,u_0\rangle-\frac{1}{2}\|u_0\|^2+\frac{1}{2}\int_{\mathbb{R}^3}(V(x)-V_{\infty})(u_n-u_0)^2\,{\rm d}x\\
&+\frac{1}{4}\Big(\int_{\mathbb{R}^3}(\phi_{u_n}^su_n^2-\phi_{v_n^1}^s(v_n^1)^2)\,{\rm d}x\Big)-\frac{\lambda}{p+1}\Big(\int_{\mathbb{R}^3}(|u_n|^{p+1}-|v_n^1|^{p+1})\,{\rm d}x\Big)\\
&=I_{V,\lambda}(u_0)+\langle u_n-u_0,u_0\rangle+\frac{1}{2}\int_{\mathbb{R}^3}(V(x)-V_{\infty})(u_n-u_0)^2\,{\rm d}x\\
&+\frac{1}{4}\Big(\int_{\mathbb{R}^3}(\phi_{u_n}^su_n^2-\phi_{v_n^1}^s(v_n^1)^2-\phi_{u_0}^su_0^2)\,{\rm d}x\Big)\\
&-\frac{\lambda}{p+1}\Big(\int_{\mathbb{R}^3}(|u_n|^{p+1}-|v_n^1|^{p+1}-|u_0|^{p+1})\,{\rm d}x\Big)\\
\end{align*}
In view of $(V_2)$ and the locally compactness of Sobolev embedding, we have
\begin{equation*}
\int_{\mathbb{R}^3}(V(x)-V_{\infty})(u_n-u_0)^2\,{\rm d}x\rightarrow0.
\end{equation*}
Thus, using \eqref{equ4-9} and \eqref{equ4-10}, we conclude that
\begin{equation*}
I_{V,\lambda}(u_n)-I_{V,\lambda}(u_0)-I_{\lambda}^{\infty}(v_n^1)\rightarrow0.
\end{equation*}

Step 4. Let us set $v_n^2(\cdot):=v_n^1(\cdot)-w^1(\cdot-y_n^1)$, then $v_n^2\rightharpoonup0$ in $H^s(\mathbb{R}^3)$. From the Brezis-Lieb Lemma and Lemma \ref{lem2-2}, we get
\begin{align}\label{equ4-11}
\|v_n^2\|^2&=\|u_n-u_0-w^1(\cdot-y_n^1)\|^2=\|u_n-u_0\|^2-\|w^1(\cdot-y_n^1)\|^2+o_n(1)\nonumber\\
&=\|u_n\|^2-\|u_0\|^2-\|w^1(\cdot-y_n^1)\|^2+o_n(1),
\end{align}
\begin{align}\label{equ4-12}
\|v_n^2\|_{p+1}^{p+1}&=\|u_n-u_0-w^1(\cdot-y_n^1)\|_{p+1}^{p+1}=\|u_n-u_0\|_{p+1}^{p+1}-\|w^1(\cdot-y_n^1)\|_{p+1}^{p+1}+o_n(1)\nonumber\\
&=\|u_n\|_{p+1}^{p+1}-\|u_0\|_{p+1}^{p+1}-\|w^1\|_{p+1}^{p+1}+o_n(1),
\end{align}
\begin{align}\label{equ4-13}
&\int_{\mathbb{R}^3}\phi_{v_n^2}^s(v_n^2)^2\,{\rm d}x=\int_{\mathbb{R}^3}\phi_{v_n^1}^s(v_n^1)^2\,{\rm d}x-\int_{\mathbb{R}^3}\phi_{w^1(x-y_n^1)}^s(w^1(x-y_n^1))^2\,{\rm d}x+o_n(1)\nonumber\\
&=\int_{\mathbb{R}^3}\phi_{u_n}^su_n^2\,{\rm d}x-\int_{\mathbb{R}^3}\phi_{u_0}^su_0^2\,{\rm d}x-\int_{\mathbb{R}^3}\phi_{w^1}^s(w^1)^2\,{\rm d}x+o_n(1),
\end{align}
\begin{align}\label{equ4-14}
&\int_{\mathbb{R}^3}\phi_{v_n^2}^sv_n^2\varphi\,{\rm d}x=\int_{\mathbb{R}^3}\phi_{v_n^1}^sv_n^1\varphi\,{\rm d}x-\int_{\mathbb{R}^3}\phi_{w^1(x-y_n^1)}^sw^1(x-y_n^1)\varphi\,{\rm d}x+o_n(1)\nonumber\\
&=\int_{\mathbb{R}^3}\phi_{u_n}^su_n\varphi\,{\rm d}x-\int_{\mathbb{R}^3}\phi_{u_0}^su_0\varphi\,{\rm d}x-\int_{\mathbb{R}^3}\phi_{w^1(x-y_n^1)}^sw^1(x-y_n^1)\varphi\,{\rm d}x+o_n(1)
\end{align}
for any $\varphi\in (H^s(\mathbb{R}^3))'$ and
\begin{align}\label{equ4-15}
&\int_{\mathbb{R}^3}V(x)|v_n^2|^2\,{\rm d}x=\int_{\mathbb{R}^3}V(x)|v_n^1|^2\,{\rm d}x-\int_{\mathbb{R}^3}V(x)|w^1(x-y_n^1)|^2\,{\rm d}x+o_n(1)\\
&=\int_{\mathbb{R}^3}V(x)|u_n|^2\,{\rm d}x-\int_{\mathbb{R}^3}V(x)|u_0|^2\,{\rm d}x-\int_{\mathbb{R}^3}V(x)|w^1(x-y_n^1)|^2\,{\rm d}x+o_n(1).
\end{align}
By \eqref{equ4-11}-\eqref{equ4-15}, we can similarly check that
\begin{equation}\label{equ4-16}
\left\{
  \begin{array}{ll}
    I_{V,\lambda}(v_n^2)=I_{V,\lambda}(u_n)-I_{V,\lambda}(u_0)-I_{\lambda}^{\infty}(w^1)+o_n(1), & \hbox{} \\
    I_{\lambda}^{\infty}(v_n^2)=I_{V,\lambda}(v_n^1)-I_{\lambda}^{\infty}(w^1)+o_n(1), & \hbox{} \\
    \langle I_{V,\lambda}'(v_n^2),\varphi\rangle=\langle I_{V,\lambda}'(u_n),\varphi\rangle-\langle I_{V,\lambda}'(u_0),\varphi\rangle-\langle (I_{\lambda}^{\infty})'(w^1),\varphi\rangle+o_n(1)=o_n(1). & \hbox{}
  \end{array}
\right.
\end{equation}
Hence, by using \eqref{equ4-8}, we have
\begin{equation*}
I_{V,\lambda}(u_n)=I_{V,\lambda}(u_0)+I_{\lambda}^{\infty}(v_n^1)+o_n(1)=I_{V,\lambda}(u_0)+I_{\lambda}^{\infty}(w^1)+I_{\lambda}^{\infty}(v_n^2)+o_n(1).
\end{equation*}
From Lemma \ref{lem3-2}, we see that any critical point of $I_{\lambda}^{\infty}$ is less than zero and so $I_{\lambda}^{\infty}(w^1)\geq0$, and from Step 1, we know that $I_{V,\lambda}(u_0)\geq0$, thus it follows that
\begin{equation*}
I_{V,\lambda}(v_n^2)=I_{V,\lambda}(u_n)-I_{V,\lambda}(u_0)-I_{\lambda}^{\infty}(w^1)+o_n(1)\leq c_{\lambda}.
\end{equation*}

Similar to the argument in Step 2, let
\begin{equation*}
\delta_1=\limsup_{n\rightarrow+\infty}\sup_{y\in\mathbb{R}^3}\int_{B_1(y)}|v_n^2|^2\,{\rm d}x.
\end{equation*}
If vanishing occurs, then $\|v_n^2\|\rightarrow0$, i.e. $\|u_n-u_0-w^1(\cdot-y_n^1)\|\rightarrow0$ and thus Lemma \ref{lem4-2} holds with $k=1$. If nonvanishing occurs, then there exists a sequence $\{y_n^2\}\subset\mathbb{R}^3$ and $w^2\in H^s(\mathbb{R}^3)$ such that $\widetilde{v_n^2}(x)\:= v_n^2(x+y_n^2)\rightharpoonup w^2$ in $H^s(\mathbb{R}^3)$. By \eqref{equ4-16}, we have that $(I_{\lambda}^{\infty})'(w^2)=0$. Furthermore, $v_n^2\rightharpoonup0$ in $H^s(\mathbb{R}^3)$ implies that $|y_n^2|\rightarrow+\infty$ and $|y_n^1-y_n^2|\rightarrow+\infty$.
By iterating this procedure we obtain sequences of points $\{y_n^k\}\subset\mathbb{R}^3$ such that $|y_n^k|\rightarrow+\infty$ and $|y_n^k-y_n^{k'}|\rightarrow+\infty$ for $k\neq k'$ and $v_n^k=v_n^{k-1}-w^{k-1}(\cdot-y_n^{k-1})$ with $k\geq2$ such that
\begin{equation*}
v_n^k\rightharpoonup0\,\,\text{in}\,\, E\quad \text{and}\,\,(I_{\lambda}^{\infty})'(w^k)=0
\end{equation*}
and
\begin{equation}\label{equ4-17}
\left\{
  \begin{array}{ll}
    \|u_n\|^2-\|u_0\|^2-\sum_{j=1}^{k-1}\|w^j(\cdot-y_n^j)\|^2=\|u_n-u_0-\sum_{j=1}^{k-1}w^j(\cdot-y_n^j)\|^2+o_n(1), & \hbox{} \\
    I_{V,\lambda}(u_n)-I_{V,\lambda}(u_0)-\sum_{j=1}^{k-1}I_{\lambda}^{\infty}(w^j)-I_{\lambda}^{\infty}(v_n^k)=o_n(1). & \hbox{}
  \end{array}
\right.
\end{equation}
Since $\{u_n\}$ is bounded in $H^s(\mathbb{R}^3)$, \eqref{equ4-17} implies that the iteration stops at some finite index $l+1$. Therefore $v_n^{l+1}\rightarrow0$ in $H^s(\mathbb{R}^3)$, by \eqref{equ4-17}, it is easy to verify that conclusions $(iv)$ and $(v)$ hold. The proof is completed.

\end{proof}
Lemma \ref{lem4-2} generalizes Lemma 3.6 in \cite{ZZ} and Theorem 5.1 in \cite{JT}. Based on Lemma \ref{lem4-2}, we can prove the $(PS)_{c_{\lambda}}$ condition of the functional $I_{V,\lambda}$ holds. That is, we have the following result.

\begin{lemma}\label{lem4-3}
Assume that $(V_1)-(V_2)$ hold and $2<p<2_s^{\ast}$, let $\{u_n\}$ be a bounded $(PS)_{c_{\lambda}}$ sequence for $I_{V,\lambda}$. Then up to a subsequence, $\{u_n\}$ converges to a nontrivial critical point $u_{\lambda}$ of $I_{V,\lambda}$ with $I_{V,\lambda}(u_{\lambda})=c_{\lambda}$ for any $\lambda\in\Lambda$.
\end{lemma}
\begin{proof}
For any $\lambda\in\Lambda$ fixed. Let $u_{\lambda}^{\infty}$ be the minimizer of $c_{\lambda}^{\infty}$, by Lemma \ref{lem3-2}, we have that $I_{\lambda}^{\infty}(u_{\lambda}^{\infty})=\max_{\theta\geq0}I_{\lambda}^{\infty}(\theta^{2s}u_{\lambda}^{\infty}(\theta x))$. Then choosing $v(x)=\theta^{2s}u_{\lambda}^{\infty}(\theta x)$ for $\theta$ large enough in Lemma \ref{lem4-1}, by $(V_2)$, we have that
\begin{equation}\label{equ4-18}
c_{\lambda}\leq\max_{\theta\geq0}I_{V,\lambda}(\theta^{2s}u_{\lambda}^{\infty}(\theta x))<\max_{\theta\geq0}I_{\lambda}^{\infty}(\theta^{2s}u_{\lambda}^{\infty}(\theta x))=I_{\lambda}^{\infty}(u_{\lambda}^{\infty})=c_{\lambda}^{\infty}.
\end{equation}
By Lemma \ref{lem4-2}, there exists $l\in\mathbb{N}\cup\{0\}$ and $\{y_n^k\}\subset\mathbb{R}^3$ with $|y_n^k|\rightarrow+\infty$ for each $1\leq k\leq l$, and $u_{\lambda}\in H^s(\mathbb{R}^3)$, $w^k\in H^s(\mathbb{R}^3)$ such that
\begin{equation*}
I_{V,\lambda}'(u_{\lambda})=0,\,\, u_n\rightharpoonup u_{\lambda},\,\, I_{V,\lambda}(u_n)\rightarrow I_{V,,\lambda}(u_{\lambda})+\sum_{k=1}^lI_{\lambda}^{\infty}(w^k),
\end{equation*}
where $w^j$ is critical point of $I_{\lambda}^{\infty}$ for $1\leq k\leq l$. Set
\begin{equation*}
a_{\lambda}=\int_{\mathbb{R}^3}|(-\Delta)^{\frac{s}{2}}u_{\lambda}|^2\,{\rm d}x,\,\, b_{\lambda}=\int_{\mathbb{R}^3}V(x)|u_{\lambda}|^2\,{\rm d}x,\,\,c_{\lambda}=\int_{\mathbb{R}^3}\phi_{u_{\lambda}}^su_{\lambda}^2\,{\rm d}x
\end{equation*}
\begin{equation*}
d_{\lambda}=\frac{\lambda}{p+1}\int_{\mathbb{R}^3}|u_{\lambda}|^{p+1}\,{\rm d}x,\,\, e_{\lambda}=\int_{\mathbb{R}^3}(x,\nabla V(x))|u_{\lambda}|^2\,{\rm d}x.
\end{equation*}
Then, we have that
\begin{equation*}
\left\{
  \begin{array}{ll}
    \frac{1}{2}a_{\lambda}+\frac{1}{2}b_{\lambda}+\frac{1}{4}c_{\lambda}-d_{\lambda}=I_{V,\lambda}(u_{\lambda}), & \hbox{} \\
    a_{\lambda}+b_{\lambda}+c_{\lambda}-(p+1)d_{\lambda}=0, & \hbox{} \\
    \frac{3-2s}{2}a_{\lambda}+\frac{3}{2}b_{\lambda}+\frac{3+2s}{4}c_{\lambda}-3d_{\lambda}+\frac{1}{2}e_{\lambda}=0. & \hbox{}
  \end{array}
\right.
\end{equation*}
Similarly to the arguments about \eqref{equ4-3}, we get
\begin{equation*}
(6s-3)I_{V,\lambda}(u_{\lambda})=\frac{1}{2}(2sb_{\lambda}+e_{\lambda})+2s(p-2)d_{\lambda}\geq0.
\end{equation*}
Thus, if $l\neq0$, we have
\begin{equation*}
c_{\lambda}=\lim_{n\rightarrow\infty}I_{V,\lambda}(u_n)=I_{V,\lambda}(u_{\lambda})+\sum_{k=1}^lI_{\lambda}^{\infty}(w^k)\geq c_{\lambda}^{\infty}
\end{equation*}
which contradicts with \eqref{equ4-18}. Hence, $l\equiv0$, from Lemma \ref{lem4-2}, we see that $u_n\rightarrow u_{\lambda}$ and $c_{\lambda}=I_{V,\lambda}(u_{\lambda})$.

\end{proof}
{\bf Proof of Theorem \ref{thm1-1}.}
From Lemma \ref{lem4-1}, it follows that for a.e. $\lambda\in\Lambda$, there exists a nontrivial critical point $u_{\lambda}\in H^s(\mathbb{R}^3)$ for $I_{V,\lambda}$ and $I_{V,\lambda}=c_{\lambda}$. We choose a sequence $\lambda_n\in[\delta,1]$ satisfying $\lambda_n\rightarrow1$, then there has a sequence of nontrivial critical points $\{u_{\lambda_n}\}$ of $I_{V,\lambda_n}$ and $I_{V,\lambda_n}(u_{\lambda_n})=c_{\lambda_n}$. Next we shall prove that $\{u_{\lambda_n}\}$ is bounded in $H^s(\mathbb{R}^3)$. Set
\begin{equation*}
a_{\lambda_n}=\int_{\mathbb{R}^3}|(-\Delta)^{\frac{s}{2}}u_{\lambda_n}|^2\,{\rm d}x,\,\, b_{\lambda_n}=\int_{\mathbb{R}^3}V(x)|u_{\lambda_n}|^2\,{\rm d}x,\,\,c_{\lambda_n}=\int_{\mathbb{R}^3}\phi_{u_{\lambda_n}}^su_{\lambda_n}^2\,{\rm d}x
\end{equation*}
\begin{equation*}
d_{\lambda_n}=\frac{\lambda_n}{p+1}\int_{\mathbb{R}^3}|u_{\lambda_n}|^{p+1}\,{\rm d}x,\,\, e_{\lambda_n}=\int_{\mathbb{R}^3}(x,\nabla V(x))|u_{\lambda_n}|^2\,{\rm d}x.
\end{equation*}
Then, we have that
\begin{equation}\label{equ4-19}
\left\{
  \begin{array}{ll}
    \frac{1}{2}a_{\lambda_n}+\frac{1}{2}b_{\lambda_n}+\frac{1}{4}c_{\lambda_n}-d_{\lambda_n}=I_{V,\lambda_n}(u_{\lambda_n}), & \hbox{} \\
    a_{\lambda_n}+b_{\lambda_n}+c_{\lambda_n}-(p+1)d_{\lambda_n}=0, & \hbox{} \\
    \frac{3-2s}{2}a_{\lambda_n}+\frac{3}{2}b_{\lambda_n}+\frac{3+2s}{4}c_{\lambda_n}-3d_{\lambda_n}+\frac{1}{2}e_{\lambda_n}=0. & \hbox{}
  \end{array}
\right.
\end{equation}
Similarly to the arguments about \eqref{equ4-3}, we get
\begin{equation*}
\frac{1}{2}(2sb_{\lambda_n}+e_{\lambda_n})+2s(p-2)d_{\lambda_n}=(6s-3)I_{V,\lambda_n}(u_{\lambda_n})\leq(6s-3)c_{\delta}
\end{equation*}
and
\begin{equation*}
\frac{1}{4}(a_{\lambda_n}+b_{\lambda_n})-(1-\frac{4}{p+1})d_{\lambda_n}=c_{\lambda_n}\leq c_{\delta}.
\end{equation*}
In view of $(V_1)$, we can deduce that $a_{\lambda_n}+b_{\lambda_n}$ is bounded, that is, $\{u_{\lambda_n}\}$ is bounded in $H^s(\mathbb{R}^3)$. Therefore, using the fact that the map $\lambda\rightarrow c_{\lambda}$ is left continuous, we have that
\begin{equation*}
\lim_{n\rightarrow\infty}I_{V,1}(u_{\lambda_n})=\lim_{n\rightarrow\infty}\Big(I_{V,\lambda_n}(u_{\lambda_n})+(\lambda_n-1)\int_{\mathbb{R}^3}\frac{1}{p+1}|u_{\lambda_n}|^{p+1}\,{\rm d }x\Big)=\lim_{n\rightarrow\infty}c_{\lambda_n}=c_1
\end{equation*}
and
\begin{equation*}
\lim_{n\rightarrow\infty}\langle I_{V,1}'(u_{\lambda_n}),\varphi\rangle=\lim_{n\rightarrow\infty}\Big(\langle I_{V,\lambda_n}'(u_{\lambda_n}),\varphi\rangle+(\lambda_n-1)\int_{\mathbb{R}^3}|u_{\lambda_n}|^{p-1}u_{\lambda_n}\varphi\,{\rm d }x=0.
\end{equation*}
These show that $\{u_{\lambda_n}\}$ is a bounded $(PS)_{c_1}$ sequence for $I_{V}:=I_{V,1}$. Then by Lemma \ref{lem4-3}, there exists a nontrivial critical point $u\in H^s(\mathbb{R}^3)$ for $I_{V}$ and $I_{V}(u)=c_1$.

Finally, we prove the existence of a ground state solution for problem \eqref{main}. Set
\begin{equation*}
m=\inf\{I_{V}(u)\,\,\Big|\,\, u\neq0,\,\, I_{V}'(u)=0\}.
\end{equation*}
Then $0\leq m\leq I_{V}(u)=c_1<+\infty$. We rule out the case $m=0$. By contradiction, let $\{u_n\}$ be a $(PS)_0$ sequence for $I_{V}$. Hence,
\begin{equation*}
0=\langle I_{V}'(u_n),u_n\rangle\geq\frac{1}{2}\|u_n\|^2-\frac{1}{p+1}\int_{\mathbb{R}^3}|u_n|^{p+1}\,{\rm d}x
\end{equation*}
which implies that
\begin{equation}\label{equ4-20}
\|u_n\|\geq C>0 \quad \text{for all}\,\, n\in\mathbb{N}.
\end{equation}
Since $I_{V}'(u_n)=0$ for any $n\in\mathbb{N}$, by Proposition \ref{pro2-1} and $(V_1)$, we have
\begin{align*}
(6s-3)I_{V}(u_n)&=(6s-3)I_{V}(u_n)-\Big(2s\langle I_{V}'(u_n),u_n\rangle-\mathcal{P}_{V}(u_n)\Big)\\
&=\frac{2s(p-2)}{p+1}\int_{\mathbb{R}^3}|u_n|^{p+1}\,{\rm d}x+\frac{1}{2}\int_{\mathbb{R}^3}(2sV(x)+(x,\nabla V(x))u_n^2\,{\rm d}x\\
&\geq\frac{2s(p-2)}{p+1}\int_{\mathbb{R}^3}|u_n|^{p+1}\,{\rm d}x,
\end{align*}
where $\mathcal{P}_V$ is defined by
\begin{align*}
\mathcal{P}_{V}(u)&=\frac{3-2s}{2}\int_{\mathbb{R}^3}|(-\Delta)^{\frac{s}{2}}u|^2\,{\rm d}x+\frac{3+2s}{4}\int_{\mathbb{R}^3}\phi u^2\,{\rm d}x+\frac{3}{2}\int_{\mathbb{R}^3}V(x)u^2\,{\rm d}x\\
&+\frac{1}{2}\int_{\mathbb{R}^3}u^2(x\cdot\nabla V(x))\,{\rm d}x-\frac{3}{p+1}\int_{\mathbb{R}^3}|u|^{p+1}\,{\rm d}x.
\end{align*}
Therefore, $\lim\limits_{n\rightarrow\infty}\|u_n\|_{p+1}=0$. Combining with $\langle I_{V}'(u_n),u_n\rangle=0$, it is easy to verify that $\lim\limits_{n\rightarrow\infty}\|u_n\|=0$. This contradicts with \eqref{equ4-20}.

In order to complete our proof, it suffices to prove $I_{V}$ can be achieved in $H^s(\mathbb{R}^3)$. Let $\{u_n\}$ be a sequence of nontrivial critical points of $I_{V}$ satisfying $I_{V}'(u_n)=0$ and $I_{V}(u_n)\rightarrow m$. Since $I_{V}(u_n)$ is bounded, by the similar arguments as \eqref{equ4-19}, we can conclude that $\{u_n\}$ is bounded in $H^s(\mathbb{R}^3)$ and as a result, $\{u_n\}$ is a $(PS)_m$ sequence of $I_{V}$. Similar arguments in Lemma \ref{lem4-3}, there exists a nontrivial $u\in H^s(\mathbb{R}^3)$ such that $I_{V}(u)=m$.

\section{Nonexistence results}
In this section, under the assumptions $(V_1)$-$(V_2)$ on the potential $V(x)$, we will establish some nonexistence results for system \eqref{main}, for the case $1\leq p\leq2$ and $p=2_s^{\ast}-1$ ,respectively.

{\bf Proof of Theorem \ref{thm1-4}.}
Suppose that $(u,\phi)\in H^s(\mathbb{R}^3)\times\mathcal{D}^{s,2}(\mathbb{R}^3)$ is a solution of \eqref{main-5}. Multiplying the first equation by $u$ and integrate over $\mathbb{R}^3$, we have that
\begin{equation}\label{equ5-1}
\int_{\mathbb{R}^3}|(-\Delta)^{\frac{s}{2}}u|^2\,{\rm d}x+\int_{\mathbb{R}^3}u^2\,{\rm d}x+\int_{\mathbb{R}^3}\phi u^2\,{\rm d}x=\int_{\mathbb{R}^3}|u|^{p+1}\,{\rm d}x.
\end{equation}
Multiplying the second equation of \eqref{main-5} by $\phi$ and integrate over $\mathbb{R}^3$, from Plancherel Theorem and $\phi=\overline{\phi}$, we get
\begin{align}\label{equ5-2}
\int_{\mathbb{R}^3}\phi u^2\,{\rm d}x&=\int_{\mathbb{R}^3}(-\Delta)^s\phi \phi\,{\rm d}x=\int_{\mathbb{R}^3}\mathcal{F}^{-1}(|\xi|^{2s}\widehat{\phi})\overline{\phi}\,{\rm d}x=\int_{\mathbb{R}^3}(|\xi|^{2s}\widehat{\phi})\overline{\widehat{\phi}}\,{\rm d}\xi\nonumber\\
&=\int_{\mathbb{R}^3}|(-\Delta)^{\frac{s}{2}}\phi|^2\,{\rm d}x.
\end{align}
Similarly, multiplying the second equation of \eqref{main-5} by $|u|$, we deduce that
\begin{align}\label{equ5-3}
\int_{\mathbb{R}^3}|u|^3\,{\rm d}x&=\int_{\mathbb{R}^3}(-\Delta)^s\phi |u|\,{\rm d}x=\int_{\mathbb{R}^3}|\xi|^{2s}\widehat{\phi}\overline{\widehat{|u|}}\,{\rm d}\xi\nonumber\\
&=\int_{\mathbb{R}^3}(\sqrt{2}|\xi|^{s}\overline{\widehat{|u|}})(\frac{1}{\sqrt{2}}|\xi|^s\widehat{\phi})\,{\rm d}\xi\nonumber\\
&=\frac{1}{2}\Big(\int_{\mathbb{R}^3}(\sqrt{2}|\xi|^{s}\overline{\widehat{|u|}})(\frac{1}{\sqrt{2}}|\xi|^s\widehat{\phi})+(\sqrt{2}|\xi|^{s}\widehat{|u|})(\frac{1}{\sqrt{2}}|\xi|^s\overline{\widehat{\phi}})\,{\rm d}\xi\Big)\nonumber\\
&\leq\frac{1}{2}\Big(\int_{\mathbb{R}^3}(\sqrt{2}|\xi|^{s}\overline{\widehat{|u|}})^2\,{\rm d}\xi+\int_{\mathbb{R}^3}(\frac{1}{\sqrt{2}}|\xi|^{s}\widehat{\phi})^2\,{\rm d}\xi\Big)\nonumber\\
&=\int_{\mathbb{R}^3}|(\Delta)^{\frac{s}{2}}u|^2\,{\rm d}x+\frac{1}{4}\int_{\mathbb{R}^3}|(\Delta)^{\frac{s}{2}}\phi|^2\,{\rm d}x,
\end{align}
where we have used the element fact that $z_1\overline{z_2}+\overline{z_1}z_2\leq|z_1|^2+|z_2|^2$ if $z_1\overline{z_2}=\overline{z_1}z_2$ for any $z_1,z_2\in \mathbb{C}$.

By \eqref{equ5-1}-\eqref{equ5-3}, we deduce that
\begin{align*}
0&=\int_{\mathbb{R}^3}|(-\Delta)^{\frac{s}{2}}u|^2\,{\rm d}x+\int_{\mathbb{R}^3}u^2\,{\rm d}x+\lambda\int_{\mathbb{R}^3}\phi u^2\,{\rm d}x-\int_{\mathbb{R}^3}|u|^{p+1}\,{\rm d}x\\
&\geq\int_{\mathbb{R}^3}|(-\Delta)^{\frac{s}{2}}u|^2\,{\rm d}x+\int_{\mathbb{R}^3}u^2\,{\rm d}x+\frac{1}{4}\int_{\mathbb{R}^3}|(-\Delta)^{\frac{s}{2}}\phi|^2\,{\rm d}x-\int_{\mathbb{R}^3}|u|^{p+1}\,{\rm d}x\\
&\geq\int_{\mathbb{R}^3}(|u|^3+u^2-|u|^{p+1})\,{\rm d}x.
\end{align*}
Since the function $t^2+t^3-t^{p+1}$ for $1<p\leq2$ is nonnegative for all $t\geq0$ and equal to zero only if $t=0$. Hence $u\equiv0$. The proof of Theorem \ref{thm1-4} is completed.

{\bf Proof of Theorem \ref{thm1-5}.}
We assume that $(u,\phi)\in H^s(\mathbb{R}^3)\times\mathcal{D}^{s,2}(\mathbb{R}^3)$ is a solution of the problem \eqref{main}, then $(u,\phi)$stisfies the following Pohozaev identity:
\begin{align*}
&(3-2s)\int_{\mathbb{R}^3}|(-\Delta)^{\frac{s}{2}}u|^2\,{\rm d}x+3\int_{\mathbb{R}^3}V(x)u^2\,{\rm d}x+\frac{3+2s}{2}\int_{\mathbb{R}^3}\phi u^2\,{\rm d}x\\
&+\int_{\mathbb{R}^3}(x,\nabla V(x))u^2\,{\rm d}x=\frac{6}{2_s^{\ast}}\int_{\mathbb{R}^3}|u|^{2_s^{\ast}}\,{\rm d}x.
\end{align*}
Multiplying the first equation of \eqref{main} by $u$ and integrating over $\mathbb{R}^3$, we get
\begin{equation*}
\int_{\mathbb{R}^3}|(-\Delta)^{\frac{s}{2}}u|^2\,{\rm d}x+\int_{\mathbb{R}^3}V(x)u^2\,{\rm d}x+\int_{\mathbb{R}^3}\phi u^2\,{\rm d}x=\int_{\mathbb{R}^3}|u|^{2_s^{\ast}}\,{\rm d}x.
\end{equation*}
Therefore, from the above two equations, we deduce that
\begin{align*}
\frac{1}{3-2s}\int_{\mathbb{R}^3}\Big(2sV(x)+(x,\nabla V(x))\Big)u^2\,{\rm d}x+\frac{6s-3}{2(3-2s)}\int_{\mathbb{R}^3}\phi u^2\,{\rm d}x=0.
\end{align*}
By the second equation in \eqref{main}, we have
\begin{equation*}
\int_{\mathbb{R}^3}|(\Delta)^{\frac{s}{2}}\phi|^2\,{\rm d}x=\int_{\mathbb{R}^3}\phi u^2\,{\rm d}x,
\end{equation*}
combining with the above equation, by $(V_1)-(V_2)$, we deduce that
\begin{equation*}
\int_{\mathbb{R}^3}\Big(2sV(x)+(x,\nabla V(x))\Big)u^2\,{\rm d}x=0,\quad \int_{\mathbb{R}^3}\phi u^2\,{\rm d}x=\int_{\mathbb{R}^3}|(\Delta)^{\frac{s}{2}}\phi|^2\,{\rm d}x=0
\end{equation*}
which implies that $u=\phi=0$. This complete the proof of Theorem \ref{thm1-5}.

\section*{Acknowledgments}

This work is supported by NSFC grant 11501403, and by the Shanxi Province Science Foundation for Youths under grant 2013021001-3.

%\bibliographystyle{elsarticle-harv}



\end{document}